\documentclass[10pt]{amsart}

 \usepackage{hyperref}
\usepackage{tikz}
\usepackage[active]{srcltx}
\usepackage{xypic}
\usepackage[all]{xy}
\usepackage{graphicx}
\usepackage{fixltx2e}
\usepackage{todonotes}
\usepackage{fullpage}

\newtheorem{theorem}{Theorem}[section]
\newtheorem{lemma}[theorem]{Lemma}
\newtheorem{proposition}[theorem]{Proposition}
\newtheorem{corollary}[theorem]{Corollary}

\newtheorem{definition}[theorem]{Definition}

\theoremstyle{remark}
\newtheorem{example}[theorem]{\bf Example}

\theoremstyle{plain}

\newtheorem{remark}[theorem]{Remark}

\newcommand{\x}{\boldsymbol{x}}

\newcommand{\PP}{\mathbb{P}}
\newcommand{\g}{\mathcal{G}}

\title[Shannon's theorem for locally compact groups]{Shannon's theorem for locally compact groups} 

\author{Behrang Forghani}
\address{Bowdoin College, 8600 College Station, Brunswick, ME 04011, USA}
\email{bforghan@bowdoin.edu}
\author{Giulio Tiozzo}
\address{University of Toronto, 40 St George St, Toronto,  ON M5S 2E4, Canada}
\email{tiozzo@math.utoronto.ca}

\begin{document}

\subjclass[2010]{60J50, 60G50, 22D05.}

\maketitle

\begin{abstract}
We establish the ray and strip approximation criteria for the identification of the Poisson boundary of random walks on locally compact
groups. This settles a conjecture from the 1990's by Kaimanovich, who formulated and proved the criterion for discrete groups. 
The key result is the proof of a version of the Shannon-McMillan-Breiman theorem for locally compact groups. 
We provide several applications to locally compact groups of isometries of nonpositively curved spaces, as well as Diestel-Leader graphs and horocylic products. 
\end{abstract}

\section{Introduction}

In the theory of random walks on groups, one of the most natural invariants is the \emph{Poisson boundary}. 
Let $G$ be a locally compact, second countable group and $\mu$ a Borel probability measure on $G$; then the Poisson boundary of the pair $(G, \mu)$ is a measure $G$-space $(\partial G, \nu)$ that represents all possible asymptotic behaviours of a sample path for a random walk driven by $\mu$. 
More precisely, the Poisson boundary  can be defined as the space of ergodic components of the shift map on the path space,
and it provides a generalization of the classical \emph{Poisson representation formula} for bounded harmonic functions on the disk:
namely, the \emph{Poisson transform} 
$$f \mapsto \varphi(g) := \int_B f \ dg \nu$$
is an isomorphism 
between $L^\infty(\partial G, \nu)$ and the space $H^\infty(G, \mu)$ of bounded $\mu$-harmonic functions on the group. 

A central question in the theory of random walks on groups has been the identification of the Poisson boundary, which is originally an abstract measure space, with a more concrete boundary arising from the geometry.

In this light, Furstenberg \cite{Furstenberg-semi} proved that for an absolutely continuous measure on a connected, semisimple Lie group, its Poisson 
boundary can be realized as a quotient $G/P$, where $P$ is a minimal parabolic subgroup. For these groups, he also classified all possible measurable 
quotients of the Poisson boundary (which are now known as $\mu$-boundaries). 

This can also be applied to lattices in semisimple Lie groups. In a more general setting, Avez \cite{avez}, Derriennic \cite{Derriennic}, and Kaimanovich-Vershik \cite{Kaimanovich-Vershik} formulated criteria for the triviality of the boundary in terms of the \emph{entropy} of the random walk. 

For discrete groups, Kaimanovich (\cite{Kaimanovich-maximal85}, \cite{Ka}) established geometric criteria for the identification of their Poisson boundaries with a 
geometric boundary. In particular, he proved the celebrated \emph{ray approximation} criterion, which roughly states that if geodesics track 
the random walk within sublinear distance, then the boundary of the space coincides with the Poisson boundary, as well as the more flexible \emph{strip approximation} criterion. 
These criteria have been applied widely to a large class of discrete groups. 

It is a long-standing conjecture (see e.g.  \cite{Kaimanovich-Zactions}, \cite{Kaimanovich-Woess}) that analogous criteria hold in 
the continuous case, i.e. when $G$ is a locally compact group and $\mu$ is absolutely continuous, or more generally spread-out, with respect to the Haar measure.
The goal of this paper is to establish the ray criterion and the strip criterion for arbitrary locally compact groups. 

In order to do so, let us consider a probability measure $\mu$ on $G$, and let us define a random walk on $G$ as 
$$x_n := g_1 g_2 \dots g_n$$
where $g_i$ are i.i.d. elements of $G$ with distribution $\mu$. The distribution of $x_n$ is given by the $n^{th}$-step convolution $\mu^{*n}$; 
let us denote as $\rho_n$ the density of $\mu^{*n}$ with respect to the Haar measure. The first key ingredient is the following.

\subsection{Shannon-McMillan-Breiman theorem}

In ergodic theory, the classical 
Shannon-McMillan-Breiman (SMB) theorem asserts 
for an ergodic, stationary process, 
the exponential convergence of the measure of cylinder sets at a rate equal to the entropy. 

In the context of random walks on discrete groups, a related statement has been proven by Derriennic \cite{Derriennic} and Kaimanovich-Vershik \cite{Kaimanovich-Vershik}. 
Namely, they show that  
\begin{equation} \label{E:SMB}
-\frac{1}{n} \log \mu^{*n}(x_n) \to h(\nu)
\end{equation}
almost surely and in $L^1$, if the asymptotic (Avez) entropy $h(\nu)$ is finite. (Note that it does not follow from the classical one, as the sequence $(x_n)$ is not stationary.) 

Derriennic (\cite{De}, page 278) asked about a generalization of \eqref{E:SMB}
to locally compact groups.
We prove the following version of the SMB theorem for measures with bounded density on locally compact groups. 

\begin{theorem}[Weak Shannon-McMillan-Breiman theorem] \label{thm-intro : bounded density}
Let $G$ be a locally compact second countable group and $\mu$ a Borel probability measure on $G$ absolutely continuous w.r.t. the left Haar measure, with density $\rho$. 
Let $(\partial G,\nu)$ be the Poisson boundary of the random walk $(G,\mu)$.
If $\rho$ is bounded and for any $n$ the differential entropy $H_n$ is finite, then
almost surely
$$
\liminf_n \left( -\frac{1}{n} \log \rho_n(x_n) \right) = h(\nu)
$$
where $h(\nu)$ is the Furstenberg entropy of the Poisson boundary.
\end{theorem}

For the definitions of differential entropy and Furstenberg entropy see Section \ref{S:entropies}. 
Let us remark that for countable groups, the theorem follows relatively easily from the submultiplicativity property
$\mu^{* n+m}(y) \geq \mu^{* n}(x) \mu^{*m}(x^{-1} y)$ which can be thusly interpreted: the random walk can reach $y$ by first going to $x$ and then going from $x$ to $y$. 

In the continuous case, the above sets have measure zero, and the argument breaks down (one of the signs of this 
is that the \emph{differential entropy} may be negative for continuous distributions). The technique we use here is based on the \emph{Furstenberg entropy}, together with a Borel-Cantelli argument, 
inspired by the proof of the classical SMB theorem by Algoet-Cover \cite{Algoet-Cover}. See Section \ref{S:AEP}.
Note that for countable groups, the Furstenberg entropy of the Poisson boundary coincides with the Avez entropy \cite{Kaimanovich-Vershik}.

\subsection{Geometric criteria for the Poisson boundary}

The previous result allows us to extend to locally compact groups the criteria given by Kaimanovich to identify the Poisson boundary.
We call a metric $d$ on a group $G$ \emph{temperate} if there exists $C > 0$ such that the Haar measure of any ball of radius 
$R$ is at most $C e^{CR}$.

\begin{theorem}[Ray approximation for locally compact groups] \label{ray:intro}
Let $G$ be a locally compact second countable group equipped with a temperate metric $d$, and let $\mu$ be a spread-out probability measure on $G$ with finite first moment with respect to $d$.
Let $(B, \lambda)$ be a $\mu$-boundary, and suppose that there exist maps $\pi_n : B \to G$ such that for almost every sample path $\x = (x_n)$ 
$$\lim_{n \to \infty} \frac{d(x_n, \pi_n(x_\infty^\lambda))}{n} = 0,$$
where $x_\infty^\lambda$ is the boundary point of the sample path $(x_n)$ in $B$.
Then $(B, \lambda)$ is the Poisson boundary of $(G, \mu)$. 
\end{theorem}

The theorem extends (\cite{Ka}, Theorem 5.5) to locally compact groups. Recall that in most geometric situations the map $\pi_n$ is constructed by taking a geodesic ray $\gamma$ connecting the base point to the boundary point $b$, and defining $\pi_n(b)$ as a closest group element from $\gamma(An)$ to $G$
(see \cite{Kaimanovich-hyp94}, \cite{Tiozzo-tracking}), where $A$ is the rate of escape of the random walk.

Note by comparison that in Bader-Shalom \cite{Bader-Shalom}, which generalizes the original approach of Furstenberg, 
it is required that the group generated by the support of $\mu$ be ``large" (precisely, that the action be strongly transitive). 
Our approach, instead, has the benefit of being inherited by subgroups; once the ray criterion holds for a group, it holds for any closed subgroup.
For instance, we obtain the Poisson boundary for any closed subgroup of $SL_n(\mathbb{R})$ (see Theorem \ref{T:Lie}).

Observe that in Theorem \ref{ray:intro} the hypotheses on the measure $\mu$ are more general than in Theorem \ref{thm-intro : bounded density}.
Indeed, we first establish the ray approximation for measures of bounded density (Theorem \ref{thm : ray1}), and then, using a strategy 
mentioned in \cite{Kaimanovich-Zactions}, we use a \emph{stopping time} to extend the criterion to the more general case of spread-out measures with finite first moment (see Section \ref{sec : Stopping time trick}). 

A slightly more flexible type of approximation criterion is given by the \emph{strip approximation}, where one needs to approximate the bilateral 
random walk with bilateral geodesics. Here, we establish it for locally compact groups. Recall that the \emph{reflected measure} $\check{\mu}$ on $G$ 
is defined as $\check{\mu} := \iota_* \mu$, where $\iota(g) = g^{-1}$. 
We denote as $\mathcal{P}(G)$ the set of subsets of $G$.
We prove the following generalization of (\cite{Ka}, Theorem 6.4) to locally compact groups:

\begin{theorem}[Strip approximation for locally compact groups] \label{T:intro-strip}
Let $G$ be a locally compact second countable group, and let $\mu$ be a spread-out probability measure on $G$. 
Let $(B_-,\lambda_-)$ be a $\check{\mu}$--boundary and $(B_+,\lambda_+)$ be a $\mu$--boundary. If there exists a sub-additive, temperate gauge $\g$ such that $\mu$ has finite first moment and a measurable $G$--equivariant map $S : B_- \times B_+ \to \mathcal{P}(G)$ such that  
$\lambda_-\otimes \lambda_+\{(b_-,b_+) : e \in S(b_-,b_+) \}=p>0$
and  for $\lambda_-\otimes\lambda_+$-almost every $(b_-,b_+)$ in $B_-\times B_+$
$$
\frac{1}{n} \log^+ m\left( S(b_-,b_+) \cap \g_{|x_n|} \right) \to 0
$$
in $\PP$--probability, then $(B_+,\lambda_+)$ is the Poisson boundary.
\end{theorem}

\subsection{Applications.} 

Let us now mention some concrete applications of these criteria. 
First of all, let us consider groups acting by isometries on spaces of nonpositive curvature. 

The action of a locally compact group $G$ on a metric space $X$ has \emph{bounded exponential growth} 
if there exist $C > 0$ and $o \in X$ such that for any $R > 0$ 
$$m(\{ g \in G \ : \ d(o, go) \leq R \}) \leq C e^{CR}.$$
Moreover, if $\mu$ has finite first moment on $X$, we define the \emph{rate of escape} of the random walk as 
$A := \lim_{n \to \infty} \frac{d(x_n o, o)}{n}$, where the limit exists and is constant a.s.
As a consequence of the ray approximation we obtain:

\begin{theorem} \label{T:intro-appl}
Let $G$ be a closed subgroup of the group of isometries of $X$, where $X$ is either: 
\begin{enumerate}
\item a proper Gromov hyperbolic space; or
\item a proper $CAT(0)$ metric space.
\end{enumerate}
Let $\mu$ be a spread-out probability measure on $G$ with finite first moment, and suppose that 
the action has bounded exponential growth. 
Then: 
\begin{itemize}
\item if the rate of escape $A = 0$, then the Poisson boundary of $(G, \mu)$ is trivial;
\item if $A > 0$, then the Poisson boundary of $(G, \mu)$ is $(\partial X, \nu)$, where 
$\partial X$ is the Gromov boundary of $X$ in case (1) and the visual boundary of $X$ in case (2), and $\nu$ is the hitting measure.
\end{itemize}
\end{theorem}

The corresponding statement for discrete groups is due to Kaimanovich \cite{Kaimanovich-hyp94} in case (1) and to Karlsson-Margulis \cite{KM} in case (2).
Note that this setting encompasses many classes of groups of interest in geometry and dynamics, such as automorphism 
groups of trees, Lie groups, automorphism groups of affine buildings, and affine groups. 
In Section \ref{S:app} we will see some of these applications in greater detail.

Moreover, our techniques also apply to non-hyperbolic settings such as automorphisms of Diestel-Leader graphs $D_{p,q}$ (which for $p, q \geq 3$ are neither hyperbolic nor CAT(0), and for $p \neq q$ are not quasi-isometric to any Cayley graph of a group  \cite{Eskin-F-W}).

\begin{theorem} \label{T:DL-intro}
Let $G$ be a closed subgroup of the group of automorphisms of the Diestel-Leader graph $DL_{p, q}$, and let $\mu$ be a spread-out probability measure on $G$
with finite first moment, and let $V$ be its vertical drift. Then: 
\begin{enumerate}
\item if $V = 0$, then the Poisson boundary of $(G, \mu)$ is trivial; 
\item if $V > 0$, then the Poisson boundary of $(G, \mu)$ is $\partial \mathbb{T}_p \times \{ \omega_2 \}$ with the hitting measure; 
\item  if $V < 0$, then the Poisson boundary of $(G, \mu)$ is $\{\omega_1 \} \times \partial \mathbb{T}_q$ with the hitting measure.
\end{enumerate}
\end{theorem}

For the definitions, see Sections \ref{S:DL} and \ref{S:horo}. The technique also applies to treebolic spaces, 
Sol-groups, and more general horocylic products.

\subsection{Historical remarks}

The Shannon-McMillan-Breiman theorem or \emph{asymptotic equipartition property} appeared first in the context of discrete stationary processes by work of Shannon \cite{Shannon}, McMillan  \cite{Mcmillan} ($L^1$ convergence), and Breiman \cite{Breiman} (almost sure convergence).  Later, their results have been generalized to continuous stationary process with densities by work of Moy \cite{Moy}, Perez \cite{Perez}, Kieffer \cite{Kieffer}, and Barron \cite{Barron}. 
In these results, martingale theory plays a crucial role. In the late 80's, Algoet-Cover  \cite{Algoet-Cover}  provided a new proof of almost sure convergence which relies on the Borel--Cantelli Lemma. 
Note that none of these results applies immediately, as the sequence $(x_n)$ produced by the random walk is not stationary.
Our Theorem \ref{thm-intro : bounded density} is however in the same spirit as the work of Algoet-Cover. 

Theorem \ref{ray:intro} has been announced (without proof) in \cite{Cartwright-Kaimanovich-Woess}, 
and is there applied to affine groups of trees.  
In \cite{Kaimanovich-Woess}, the authors prove the ray and strip criterion in the case of \emph{homogeneous Markov chains}, 
i.e. Markov chains on a countable set with a $G$-invariant distribution. This setting is intermediate between countable groups and general locally compact groups, being equivalent to considering random walks 
on locally compact groups where the measure $\mu$ is invariant by a compact open subgroup $K$ (\cite{Kaimanovich-Woess}, Proposition 2.9).  
Their proof uses once again the subadditivity of the $n$-step distribution (see \cite{Kaimanovich-Woess}, Theorem 4.2), which follows from having a countable state space, and which thus cannot be generalized to all measures on a locally compact group\footnote{As mentioned in  \cite{Kaimanovich-Woess}, page 324: ``On the other hand, the fact that $G$
has a countable homogeneous space $X$ still allows us to avoid technical problems
arising in the entropy theory of random walks on general locally compact groups".}. The key to our approach is a completely different proof of Theorem \ref{thm-intro : bounded density}, which holds without any countability assumption.
 
For discrete groups, ray approximation is due to \cite{Kaimanovich-maximal85}. Strip approximation for discrete groups is proven in \cite{Ka}, and the idea of strips originates from \cite{Ledrappier-Ballmann}; in the discrete case, indeed strip approximation holds under the more general condition of finite logarithmic moment. For our stopping time technique to work, though, we need finite first moment (see Remark \ref{R:fake}). 

Triviality and identification of the Poisson boundary for various classes of locally compact groups is proven, among others, by 
\cite{Choquet-Deny}, \cite{Furstenberg-semi}, \cite{Azencott}, \cite{Guivarch}, \cite{Raugi},   \cite{Cartwright-Kaimanovich-Woess},   \cite{Jaworski},  \cite{Jaworski-almost}, \cite{Bader-Shalom}, and  \cite{Brofferio}. 
Surveys are given in \cite{Babillot} and \cite{Erschler}.

\subsection*{Acknowledgements}
We are greatly indebted to Vadim Kaimanovich for suggesting the problem 
and providing valuable comments on the first version.
We also thank Alex Eskin, Anders Karlsson, and Wolfgang Woess for useful comments. 
G. T. is partially supported by NSERC and the Sloan Foundation.

\section{Furstenberg and differential entropies} \label{S:entropies}

By \emph{locally compact group} we will always mean a second countable, Hausdorff, locally compact topological group  $G$. 
Such a group admits a unique, up to scaling, left Haar measure $m$. Let $\mu$ be a probability measure on $G$ which is absolutely 
continuous with respect to $m$, and let us denote as $\rho := \frac{d \mu}{dm}$ the density. 
The $n$-fold convolutions of $\mu$ and $\rho$ are denoted by $\mu^{*n}$ and $\rho_n$, respectively. One can verify easily that $d\mu^{*n}=\rho_n \ dm$.

Let $(\Omega, \PP)$ be the space of sample paths for the random walk $(G,\mu)$, and let $(\partial G, \nu)$ be the Poisson boundary.  
A sample path is denoted by $\x=(x_n)$. A $\mu$-boundary $(B, \lambda)$ is a quotient of the Poisson boundary with respect to a $G$-invariant measurable partition. Let $(B,\lambda)$ be a $\mu$--boundary, and let the quotient map be
$$\begin{array}{c}
(\Omega,\PP)\to (B,\lambda)\\
(x_n)\to x_\infty^\lambda.
\end{array}
$$
In this case, $\lambda$ is $\mu$--stationary ($\mu*\lambda=\lambda$), therefore, for $\lambda$-almost every $\gamma$ in $B$ and any natural number $n$, we have 
\begin{equation}
\int_G \frac{dg\lambda}{d\lambda}(\gamma) \ d\mu^{*n}(g) = 1.
\end{equation}
Let $(\partial G,\nu)$ be the Poisson boundary of $(G,\mu)$ and $x_\infty:=x_\infty^\nu$.

\subsection{Furstenberg entropy}

The \emph{Furstenberg entropy} of the $\mu$--boundary $(B,\lambda)$ is defined as 
\begin{equation} \label{E:furst}
h(\lambda) := \int_\Omega \log\frac{dx_1\lambda}{d\lambda}(x^\lambda_\infty)\ d\PP(\x)= \int_{G\times B} \log\frac{dg\lambda}{d\lambda}(\gamma) \  dg\lambda(\gamma) d\mu(g).
\end{equation}

\begin{lemma} \label{L:well-defined}
The integral in \eqref{E:furst} above is always well-defined and non-negative, however it may equal $+ \infty$. 
\end{lemma}

\begin{proof}
Let us define $F : \Omega \to \mathbb{R}$ as 
$$
F(\x) :=\left( \frac{d x_1 \lambda}{d \lambda}(x_\infty^\lambda) \right)^{-1}.
$$
Note that  $F(\x) \geq 0$ and it belongs to $L^1(\PP)$, because
$$\int_\Omega \left( \frac{d x_1 \lambda}{d \lambda}(x_\infty^\lambda) \right)^{-1} \ d \PP(\x) = 
\int_{B \times G} \left( \frac{d g \lambda}{d \lambda}(\gamma) \right)^{-1} \ d g\lambda(\gamma) d\mu(g) = $$
$$ = 
\int_{B \times G} \left( \frac{d g \lambda}{d \lambda}(\gamma) \right)^{-1} \frac{d g\lambda}{d\lambda}(\gamma) \ d\lambda(\gamma) d\mu(g) = 
\int_{B \times G} 1  \ d\lambda(\gamma) d\mu(g) = 1$$
Moreover, one has $-\log F \geq F - 1$, and the function $F$ belongs to $L^1(\PP)$, hence 
$-\log F$ is bounded below by an $L^1$ function, so its negative part is integrable
and the integral of $- \log F$ may be finite or $+ \infty$. Finally, 
$$h(\lambda) = \int_\Omega - \log F \ d\PP \geq \int_{\Omega} ( F - 1) \ d\PP = 1 - 1 = 0.$$
\end{proof}

\begin{remark}
Note that in several places in the literature (i.e. \cite{Babillot}, \cite{Furman}) the definition of Furstenberg entropy is given under the hypothesis that the closed semigroup $sgr(\mu)$ generated by the support of $\mu$ is a group. This is helpful since in general one only has $g\lambda \ll \lambda$ for $g \in sgr(\mu)$, but is not true that $g^{-1} \lambda \ll \lambda$ . 
With our definition, however, we do not need this assumption. Indeed, the definition of $F$ in Lemma \ref{L:well-defined} is well posed, since  
$$\PP\left( \frac{d x_1 \lambda}{d \lambda}(x_\infty^\lambda) = 0 \right)  = \int_\Omega 1_{\left\{ \frac{d x_1 \lambda}{d \lambda}(x_\infty^\lambda) = 0 \right\}}(\x) \ d\PP(\x)=$$
$$ = \int_{B \times G} 1_{\left\{ \frac{dg \lambda}{d\lambda}(\gamma) = 0 \right\}} \ dg\lambda(\gamma) d\mu(g) = \int_{B \times G} 1_{\left\{ \frac{dg \lambda}{d\lambda}(\gamma) = 0 \right\}}  \frac{dg\lambda}{d\lambda}(\gamma) \ d\lambda(\gamma)d\mu(g) =  0.$$
\end{remark}

\begin{lemma} 
If $h(\lambda) $ is finite, then $g(\x) := \log \frac{dx_1 \lambda} {d\lambda}(x_\infty^\lambda)$ belongs to $L^1(\PP)$.
\end{lemma}
\begin{proof}
Let $F(\x):= \left(\frac{d x_1\lambda} {d \lambda}(x_\infty^\lambda)\right)^{-1}$, then $F$ belongs to $L^1(\PP)$ by Lemma \ref{L:well-defined}. Hence
$$
|-\log F(\x)| = | F(\x)-\log F(\x)  - F(\x)| \leq |F(\x)- \log F(\x)|+ |F(\x)| = 2F(\x)-\log F(\x),
$$
therefore, $\int_\Omega |-\log F(\x)|  \ d\PP(\x) \leq 2+h(\lambda).$
\end{proof}

\begin{theorem}\label{theo : Furstenberg-entropy}
Let $(B,\lambda)$ be a $\mu$--boundary. Then for $\PP$--almost every sample path $\x=(x_n)$ we have 
$$
\lim_n \frac{1}{n} \log \frac{dx_n\lambda}{d\lambda}(x^\lambda_\infty)=h({\lambda}).
$$
\end{theorem}

Note that strictly speaking we do not need $h(\lambda)$ to be finite for the limit to exist (even though, if $h(\lambda) = + \infty$, then the limit is infinite). 

\begin{proof}
Recall that the function $F$ defined in Lemma \ref{L:well-defined} belongs to $L^1(\PP)$, and $h(\lambda) = \int_\Omega -\log F \ d \PP$.
On the other hand, we  have $F - \log F \geq 0$, and if $h(\lambda)=-\int \log F \ d \PP < \infty$ then $F-\log F$ belongs to $L^1(\PP)$.   Applying Birkhoff's pointwise ergodic theorem twice to the functions $F$ and $F-\log F$ with respect to $U$,  the ergodic transformation induced by the shift on the space of increments, and taking the difference, implies that 
$$
\lim_n \frac{1}{n} \sum_{k=1}^{n} \log F \circ U^{k-1}(\x) = h(\lambda)
$$
for $\PP$--almost every sample path $\x$.  Observe that $\frac{d g_1 g_2d\lambda}{d\lambda}(\gamma)=\frac{d g_2\lambda}{d g^{-1}_1\lambda}(g^{-1}_1\gamma)$. Let $x_n= g_1 g_2\cdots g_n$ where $g_i$ has the law $\mu$, then 
$$
\frac{dx_n\lambda}{d\lambda}(x^\lambda_\infty)=\prod_{k=1}^{n} \frac{d g_k\lambda}{d\lambda}(x_{k-1}^{-1}x^\lambda_\infty)=\prod_{k=1}^{n} \frac{d(U^{k-1}\x)_1\lambda}{d\lambda}(U^{k-1}\x)^\lambda_\infty,
$$
and consequently, $  \frac{1}{n} \sum_{k=1}^{n} \log F \circ U^{k-1}(\x) = \frac{1}{n} \log \frac{dx_n\lambda}{d \lambda} \rightarrow h(\lambda)$.
\end{proof}

\subsection{Differential entropy}

Let $m$ be a left Haar measure on $G$, and let $\mu$ be a probability measure on $G$ absolutely continuous w.r.t. $m$. We denote as $\rho := \frac{d \mu}{d m}$
the density function. 
Then the \emph{differential entropy} of $\mu$ is defined as 
$$H(\mu) := - \int_G \log \rho(g) \ d \mu(g).$$
For any $n$, we define the $n^{th}$ differential entropy as the entropy of the $n^{th}$ step convolution of $\mu$. That is, 
$$
H_n:= H(\mu^{*n}) = -\int_G \log\rho_n(g) \ d\mu^{*n}(g)=-\int_\Omega \log\rho_n(x_n) \ d\PP(\x).
$$
Let us also define the \emph{mutual information} of two random variables $X, Y$ defined on the same space as 
$$I(X, Y) := \int \log \frac{dP_{(X, Y)}}{dP_X \otimes P_Y} \ dP_{(X, Y)}$$
where $P_X \otimes P_Y$ is the product of the laws of $X$ and $Y$, and $P_{(X, Y)}$ is the joint law of $(X, Y)$. 
Differential entropy satisfies the following basic properties. 

\begin{theorem}[Derriennic \cite{De}, page 253] \label{theo : Derriennic}
Let $(x_n)$ be a random walk on a locally compact group $G$ with distribution $\mu$, and let $(\partial G, \nu)$ its Poisson boundary. Then:
\begin{enumerate}
\item If  $I(x_1, x_k)$ is finite for some $k$, then the sequence $(I(x_1, x_n))_{n \geq k}$ is decreasing and 
$$
\lim_{n \to \infty} I(x_1, x_n)=\inf_{n\geq k} I(x_1, x_n)=h(\nu).
$$
Moreover, if $\mu$ is absolutely continuous with respect to the left Haar measure, we have: 
\item if $H_n$ and $H_{n-1}$ are finite, then $I(x_1, x_n)=H_n-H_{n-1}$;
\item if $H_n$ is finite for all $n$, then 
$$
\lim_{n \to \infty} I(x_1, x_n)=\lim_{n \to \infty} \frac{1}{n} H_n= h(\nu).
$$
\end{enumerate}
\end{theorem}

Let us remark that Derriennic first defines the \emph{asymptotic entropy} $\mathfrak{h}(\mu)$ of a random walk as the mutual information 
$$\mathfrak{h}(\mu) := I(x_1, \mathcal{A})$$
where $\mathcal{A}$ is the tail $\sigma$-algebra. Then, he proves that the asymptotic entropy equals the 
Furstenberg entropy $h(\nu)$ of the Poisson boundary $(\partial G, \nu)$. The theorem above summarizes both claims. 

\subsubsection{Fake differential entropy }
When $G$ is a discrete group, the differential entropy coincides with the (Shannon) entropy.  
In this case, $H_1$ being finite implies that $H_n \leq n H_1$ for all $n \geq 1$. 
This follows from the fact that in the discrete case we have $\mu^{*(n+1)}(gh) \geq \mu^{*n}(g) \mu(h)$. 
However, as the following example shows, the sub-additivity property does not hold for the differential entropy on locally compact groups
(hence, the usual proof of Shannon's theorem in the discrete case cannot carry to non-discrete groups). 

\begin{example}\label{ex : fake}
Let $\mu$ be uniformly distributed on the interval $(0,1)$, then $H(\mu)=H_1=0$, but 
$$
\rho_2(x)= \begin{cases}
x & 0<x<1 \\
2-x & 1 \leq x <2\\
0 & \textup{otherwise}
\end{cases}
$$
Therefore, $H(\mu^{*2})=H_2= 1/2$, and consequently, $H(\mu^{*2}) \not \leq 2 H(\mu)=0$. 
\end{example}

\section{The Shannon-McMillan-Breiman theorem} \label{S:AEP}

The goal of this section is to prove our key result, namely a weak form of the Shannon-McMillan-Breiman theorem for locally compact groups (Theorem \ref{thm : bounded density}). 
We will start with a simple lemma (which appears in  \cite{Algoet-Cover}). 
\begin{lemma}\label{Lem : Borel}
Let $(f_n)$ be a sequence of positive measurable functions on a probability space $(M, \kappa)$ such that $\int f_n \ d\kappa \leq C$ is uniformly bounded. 
Then 
$$
\limsup_n\frac{1}{n} \log{f_n(x)} \leq 0 \ \ \kappa\mbox{-a.e.}
$$
\end{lemma}
\begin{proof}
For any $\epsilon>0$, we have  
$$
\kappa(\left\{x: f_n(x)>e^{n\epsilon}\right\}) \leq e^{-n\epsilon} \int f_n \ d\kappa \leq C  e^{-n\epsilon}.
$$
By Borel-Cantelli, $\kappa(\left\{x: f_n(x)>e^{n\epsilon}\ \mbox{infinitely often }\right\}) =0$.
\end{proof}

In the special case of compact groups, this immediately implies the Shannon-McMillan-Breiman theorem. 
In fact, it is easy to see that the Poisson boundary of a compact group is trivial, so the asymptotic entropy 
must be zero whenever is finite, see \cite{Avez76} and \cite{De}. We obtain:

\begin{corollary}
Let $G$ be a second countable compact group. Let $\mu$ be absolutely continuous w.r.t. the Haar (probability) measure $m$ with density $\rho$. If $\rho \leq C$, then 
$$
\lim_{n\to \infty}\frac{1}n \log{\rho_n(x_n) }=0
$$
for $\PP$--almost every sample path $\x=(x_n)$. 
\end{corollary}
\begin{proof}
Note that for any $n$, we have $\rho_n(x_n) \leq C$. Therefore, 
\begin{equation}\label{Eq : compact 1}
\limsup \frac1n \log{\rho_n(x_n) } \leq 0.
\end{equation}
We have $m(G)=1$, therefore for any $n$
$$
\int_\Omega \frac{1}{\rho_n(x_n)} d\PP = \int_G dm(g) = 1.
$$
By Lemma~\ref{Lem : Borel}, we have 
\begin{equation}\label{Eq : compact 2}
\limsup - \frac1n \log{\rho_n(x_n) } \leq 0.
\end{equation}
Combining equations \eqref{Eq : compact 1} and \eqref{Eq : compact 2} implies the desired result.
\end{proof}

Now, let us continue the proof for general groups. The key estimate is the following.

\begin{lemma} \label{L:estimate}
Suppose that the density $\rho$ is bounded, i.e. $\rho(x) \leq C$ for some positive $C$ and for a.e. $x \in G$. If $(B,\lambda)$ is a $\mu$--boundary, then  
$$\int_\Omega \frac{dx_{n-1}\lambda}{d\lambda}(x^\lambda_\infty)  \rho_n(x_n)\ d\PP(\x) \leq C. $$
\end{lemma}

\begin{proof}
Since $\lambda$ is $\mu$--stationary, for any $n$
$$
\lambda = \int_G g\lambda \ d\mu^{*n}(g) = \int_G g \lambda \ \rho_n(g) \ d m(g).
$$
Hence, by taking the Radon-Nikodym derivative with respect to $\lambda$, 
\begin{equation} \label{E:stationary}
1 = \int_G \frac{d g \lambda}{d \lambda}(\gamma) \ \rho_n(g) \ d m(g) \qquad \textup{for }\lambda \textup{-a.e. } \gamma \in B,
\end{equation}
which is the same as saying 
$$1 = \int_G \frac{d g \lambda}{d \lambda}(x^\lambda_\infty) \ \rho_n(g) \ d m(g) \qquad \textup{for }\PP \textup{-a.e. } \x \in \Omega.$$
Let us now estimate 
$$\int_\Omega \frac{dx_{n-1} \lambda}{d \lambda}(x^\lambda_\infty)  \rho_n(x_n)\ d\PP(\x).$$
The integral depends on $x_{n-1}, x_n, x_\infty$, therefore 
$$\int_\Omega \frac{dx_{n-1} \lambda}{d \lambda}(x^\lambda_\infty)  \rho_n(x_n)\ d\PP(\x)= \int_\Omega \frac{d h \lambda}{d \lambda}(\gamma)  \rho_n(g)d\PP_{x_{n-1}, x_n, x_\infty}(h,g,\gamma).$$
Note that 
$$
\frac{d\PP_{x_{n-1}, x_n, x_\infty}}{dm\times dm\times d \lambda}(h,g,\gamma)= \rho_{n-1}(h)\rho(h^{-1}g)\frac{dg \lambda}{d \lambda}(\gamma).
$$
Therefore,
$$
\int_\Omega \frac{dx_{n-1} \lambda}{d \lambda}(x^\lambda_\infty)  \rho_n(x_n)\ d\PP(\x)= $$
$$ = \int_{G\times G\times B} \frac{dh \lambda}{d \lambda}(\gamma) \rho_n(g)   \rho_{n-1}(h)\rho(h^{-1}g)\frac{dg \lambda}{d \lambda}(\gamma)\ dm(h) dm(g) d \lambda(\gamma).$$
We have $\rho \leq C$, therefore
$$
\int_\Omega \frac{dx_{n-1} \lambda}{d \lambda}(x^\lambda_\infty)  \rho_n(x_n)d\PP(\x) \leq C \int_{G\times G\times B}  \rho_{n-1}(h)  \frac{dh \lambda}{d \lambda}(\gamma) \rho_n(g)   \frac{dg \lambda}{d \lambda}(\gamma) dm(h) dm(g) d \lambda(\gamma).
$$
The right hand side is equal to
$$
\int_{B} d \lambda(\gamma) \int_{G}  dm(h) \rho_{n-1}(h)  \frac{dh \lambda}{d \lambda}(\gamma) \int_G \rho_n(g)   \frac{dg \lambda}{d \lambda}(\gamma) dm(g).
$$
Now, by equation \eqref{E:stationary} the inner integral equals $1$ for $ \lambda$-a.e. $\gamma$, and the same is true for the middle integral, hence 
the total integral is $1$ and 
$$\int_\Omega \frac{dx_{n-1} \lambda}{d \lambda}(x^\lambda_\infty)  \rho_n(x_n)\ d\PP(\x) \leq C. $$
\end{proof}

\begin{proposition}\label{P : crucial}
Let $(B,\lambda)$ be a $\mu$--boundary. If $\rho$ is bounded and $h(\lambda)< +\infty$, then 
$$
\limsup_n \frac{1}{n} \log \rho_n(x_n) \leq -h(\lambda)
$$
for $\lambda$--almost every $\gamma$ in $B$ and $\PP^{\gamma}$--almost every sample path $(x_n)$.
Equivalently,
$$
\liminf_n \left( -\frac{1}{n} \log \rho_n(x_n) \right) \geq h(\lambda)\cdot
$$
\end{proposition}

\begin{proof}
By Lemma \ref{L:estimate} and Lemma \ref{Lem : Borel}, for $\PP$--a.e. $\x$ we have
$$\limsup_{n \to \infty} \frac{1}{n} \log \left( \frac{dx_{n-1}\lambda}{d\lambda}(x^\lambda_\infty)  \rho_n(x_n) \right) \leq 0. $$
Moreover, by  Theorem \ref{theo : Furstenberg-entropy}, 
$$\lim_n \frac1n \log\frac{dx_n\lambda}{d\lambda}(x^\lambda_\infty)=h(\lambda)$$ 
which implies the desired result.
\end{proof}

\begin{theorem}[Weak Shannon-McMillan-Breiman theorem] \label{thm : bounded density}
Let $G$ be a locally compact group and $\mu$ a probability measure on $G$ absolutely continuous w.r.t. the left Haar measure, with density $\rho$. 
Let $(\partial G,\nu)$ be the Poisson boundary of the random walk $(G,\mu)$.
If $\rho$ is bounded and for any $n$ the differential entropy $H_n$ is finite, then
almost surely
$$
\liminf_n \left( -\frac{1}{n} \log \rho_n(x_n) \right) = h(\nu).
$$
\end{theorem}

\begin{proof}
Let $C := \Vert \rho \Vert_\infty < + \infty$. Note that by taking convolutions one also shows $\rho_n \leq C$ for all $n$, which implies $-\log\rho_n(x_n) \geq -\log C$. Therefore, we can apply Fatou's lemma to the part 3 of Theorem~\ref{theo : Derriennic};
$$
\int_\Omega \liminf_n \left( -\frac{1}{n} \log\rho_n(x_n) \right) d\PP(\x) \leq \lim_n \int_\Omega \left( -\frac{1}{n} \log\rho_n(x_n) \right) d\PP(\x) = h(\nu). $$
On the other hand, Proposition \ref{P : crucial} implies that 
$$
\int_\Omega \liminf_n\left(- \frac{1}{n} \log \rho_n(x_n) \right) d\PP\geq h(\nu)\cdot
$$
hence equality must hold almost everywhere.
\end{proof}

Combining Proposition~\ref{P : crucial} and Theorem~\ref{liminf stopping time} yields the maximality of the Furstenberg entropy among $\mu$--boundaries, which should be compared to Theorem~\ref{thm : De ma}.
\begin{corollary}
With the same notation as in the above theorem, $h(\lambda) \leq h(\nu) $ for any $\mu$--boundary $(B,\lambda)$.
\end{corollary}

\section{Maximal entropy and the Poisson boundary}

We will now use the Furstenberg entropy and the weak Shannon theorem of the previous section to give an entropy criterion to recognize the Poisson boundary.
The starting point is the following theorem, due to Derriennic. 

\begin{theorem}[Derriennic \cite{De}, page 268] \label{thm : De ma}
Let $\mu$ be a probability measure on a locally compact group $G$, and 
let $(\partial G, \nu )$ be the Poisson boundary of the random walk $(G, \mu)$.
For any $\mu$-boundary $(B,\lambda)$, we have 
$$
h(\lambda) \leq h(\nu).
$$
Moreover, if $ h(\nu)$ is finite, then $h(\nu)=h(\lambda)$ if and only if $(B,\lambda)$ is the Poisson boundary.
\end{theorem}

Combining Derriennic's result with Theorem~\ref{thm : bounded density} yields the following criterion.

\begin{theorem}[Entropy criterion, version 1]
Let $\mu$ be a probability measure on a locally compact group $G$ with bounded density with respect to the left Haar measure, suppose that $H_n$ is finite for all $n$, 
and let $(B, \lambda)$ be a $\mu$-boundary. Then the following statements are equivalent:
\begin{enumerate}
\item the $\mu$--boundary $(B,\lambda)$ is the Poisson boundary.
\item $h(\nu)=h(\lambda)$.
\item for  $\PP$--almost every sample path $(x_n)$ one has 
$$\liminf_n -\frac{1}{n} \log \left( \rho_n(x_n) \frac{dx_n\lambda}{d\lambda}(x_\infty^\lambda)\right) = 0.$$
\end{enumerate}
\end{theorem}

In order to apply this criterion to the concrete situations, one needs to formulate it in terms of the existence of certain sets $A_{n, \epsilon}$ which ``trap" the random walk. 

\begin{theorem}[Entropy criterion, version 2] \label{thm : general An sets}
Let $\mu$ be a probability measure on a locally compact group $G$ with bounded density with respect to the left Haar measure, suppose that $H_n$ is finite for all $n$, and let $(B, \lambda)$ be a $\mu$-boundary. 
If there exists a subset $W$ of $B$ such that $\lambda(W)>0$ and for $\lambda$-almost every $\gamma$ in $W$ and any $\epsilon>0$ there exist 
measurable sets $A_{n,\epsilon}^\gamma \subset G$ with finite Haar measure such that 
\begin{enumerate}
\item $\limsup_n\PP^\gamma \{\x : x_n\in A^\gamma_{n,\epsilon} \}> 0 $ 
\item $\limsup_n\frac{1}{n} \log m(A^\gamma_{n,\epsilon})< \epsilon$,
\end{enumerate}
then $h(\nu)=h(\lambda)$; therefore, $(B,\lambda)$ is the Poisson boundary.
\end{theorem}

The proof is very similar to the discrete case, the only subtle point being that in Theorem \ref{thm : bounded density} we only have convergence for the liminf rather than the limit as in the original proof, which, however, still works under this assumption.  For the sake of completeness, it is written in the Appendix.

\section{Gauge and finite moment}

In order to apply the entropy criteria to geometric situations, we need to define distances on the group. A slightly more general object than a norm 
on a group is a \emph{gauge}, which we will discuss in this section. 

\begin{definition}
Given a group $G$, a \emph{gauge} on $G$ is a sequence $\g=(\g_n)_{n \geq 1}$ of subsets of $G$ such that $\g_n \subset \g_{n+1}$ for any $n$, and $\bigcup_n \g_n = G$.
\end{definition}

Given a gauge $\g$ and $g$ in $G$, we define  $|g|_\g:=\min\{ n: g \in \g_n\}$.
A gauge $\g$ is \emph{sub-additive} if 
$$|gh|_\g \leq |g|_\g + |h|_\g$$ 
for any $g, h \in G.$
Moreover, a family of gauges $\g^k=(\g^k_n)$ is called \emph{uniformly temperate} if 
$$
\sup_{k,n} \frac{1}{n} \log m(\g^k_n) < \infty.
$$

The most typical example of a gauge arises when the group $G$ acts by isometries on a metric space $(X, d)$. Then one fixes a base point $o \in X$ and one sets for any $n \geq 0$ 
$$\g_n := \{ g \in G \ : \ d(x, gx) \leq n \}.$$
This gauge is sub-additive by the triangle inequality.

The following two lemmas are analogous to the ones for discrete groups due to Derriennic \cite{Derriennic}, so we omit their proofs. 

\begin{lemma}\label{lem : finite moment}
Let $\g$ be a gauge on a group $G$ equipped with a probability measure $\mu$.  If $\int_G \log |g|_\g \ d\mu(g)$ is finite and we denote $B_n := \g_n \setminus \g_{n-1}$, then 
$$
-\sum_{n=1}^\infty \mu(B_n) \log \mu(B_n)
$$
is finite.
\end{lemma}

\begin{lemma}
Let $\g$ be a temperate gauge on a locally compact group $G$, and let $\mu$ be a probability measure on $G$ such that $\rho = \frac{d \mu}{dm}$ is bounded. If $\int_G  |g|_\g \ d\mu(g)$ is finite, then $H_1$ is finite.
\end{lemma}

If additionally $\g$ is sub-additive, then $\int_G |g|_\g \ d \mu^{*n}(g) \leq n \int_G |g|_\g \ d \mu(g)$, hence: 

\begin{corollary}\label{cor : Hn finite}
Let $\g$ be a temperate and sub-additive gauge on a locally compact group $G$, and let $\mu = \rho \ dm$ be a probability measure on $G$
with bounded density. If $\int_G  |g|_\g \ d\mu(g)$ is finite, then $H_n$ is finite for all $n$.
\end{corollary}

A very typical way to construct temperate gauges on a group is to take as ``unit ball" a compact subset which generates the group. 

\begin{lemma} 
Let $K \subseteq G$ be a  compact set with non-empty interior, and such that $K$ generates $G$ as a semigroup. Then there exists a natural number $\ell$ such that for any $n$
$$
m(K^{n+1}) \leq \ell m(K^n).
$$
\end{lemma}

\begin{proof}
Let $x$ be in $K$ and  $U$ be an open neighborhood of $x$ such that $U \subset K$, then 
$$
K^2 \subset \bigcup_{y\in K^2} y x^{-1} U.
$$
Since $K^2$ is compact, there exists a finite set 
$\{y_1,y_2,\ldots, y_{\ell}\} \subset K^2$ such that 
$$
K^2 \subset \bigcup_{i=1}^\ell y_i x^{-1} U \subset \bigcup_{i=1}^\ell y_i x^{-1} K\Rightarrow m(K^2) \leq \ell m(K).
$$
Therefore, $K^{n+1}\subset \left(\bigcup_{i=1}^\ell y_i x^{-1}K \right)K^{n-1}=\bigcup_{i=1}^\ell y_i x^{-1} K^n$, and consequently, 
$m(K^{n+1}) \leq \ell m(K^n)$.
\end{proof}

\begin{corollary} \label{L:temperate}
Let $K \subseteq G$ be a compact set with non-empty interior and such that $K$ generates $G$ as a semigroup.  Then  $\g := (\g_n)$
with $\g_n:= K^n$ is a temperate gauge. 
\end{corollary}

\begin{lemma} \label{L:autX}
Let $X$ be a locally bounded, infinite graph, let $G$ be a closed group of automorphisms of $X$, and let $o$ a vertex of $X$. Then 
the gauge $\g := (\g_n)$ with
$$\mathcal{G}_n := \{ g \in G  :   d(o, go) \leq n \}$$ 
is temperate.
\end{lemma}

\begin{proof}
First of all, the stabilizer $G_0 := \{ g \in G \ : \ go = o \}$ is compact (\cite{Woess-trees}, Lemma 1), so $m(G_0) < \infty$.
Moreover, let $x$ be a vertex of $X$, and pick some $g_x \in X$ such that $go = x$, if it exists. 
Then 
$$\{ g \in G \ : \ d(o, go) \leq R \} \subseteq \bigcup_{d(o,x)\leq R} g_x G_0$$
hence, since the number of vertices at distance $\leq R$ from $o$ is at most $C^R$,
$$m(\{ g \in G \ : \ d(o, go) \leq R \}) \leq \# \{x \ : \ d(o,x)\leq R \}m(G_0) \leq m(G_0) C^R.$$
\end{proof}

\section{Stopping time trick}\label{sec : Stopping time trick}
The aim of this section is to overcome the bounded density conditions of the previous sections. We will show that for a given random walk, we can 
always construct a random walk with bounded density with respect to the Haar measure whose Poisson boundary is the same as the original one. 
The new random walk will be constructed using stopping times.

\begin{definition}
Let $G$ be a locally compact second countable group with the left Haar measure $m$. A Borel probability measure $\mu$ on $G$ is called \emph{spread--out} if there exists $n \geq 1$ such that $\mu^{*n}$ is not singular with respect to $m$. 
\end{definition}

First, let us define some useful notation. For any finite Borel measure $\zeta$ and any bounded measurable function $f$ define the operator $P^\zeta$ as 
$$P^\zeta f (x) := \int_G f(xy) \ d\zeta (y).$$
The following are immediate properties of this operator for any two finite measures $\zeta_1$ and $\zeta_2$ and for any measurable $f$:

\begin{itemize}
\item $P^{\zeta_1+\zeta_2} f=P^{\zeta_1}f +P^{\zeta_2} f$,
\item $P^{\zeta_1} P^{\zeta_2} f= P^{\zeta_1 * \zeta_2} f$.
\end{itemize}
\begin{definition}
Let $f$ be a bounded measurable function.
We say $f$ is $\mu$--harmonic if $P^\mu f = f$. 
\end{definition}
Note that when $\mu$ is spread-out, any bounded $\mu$--harmonic function is also continuous, see \cite{Babillot}.

\begin{lemma}
Let $G$ be a second countable locally compact group, and let $\mu$ be a probability measure on $G$. 
Let $\mu=\alpha+\beta$, where $\alpha$ and $\beta$ are two non-zero sub-probability measures on $G$.  Then
$$
\theta :=\sum_{n=0}^\infty \beta^{*n}*\alpha
$$
and $\mu$ have the same Poisson boundary.
\end{lemma}

\begin{proof}
Note that $\theta= (\delta_e-\beta)^{-1} * \alpha$. Let $f$ be a bounded $\mu$--harmonic function. Then $P^\mu f=f$ and consequently 
$$
P^{\alpha+ \beta}f=P^\alpha f + P^\beta f = f \Leftrightarrow P^\alpha f = P^{ (\delta_e-\beta)} f \Leftrightarrow  P^{ (\delta_e-\beta)^{-1}}  P^\alpha f= P^{\theta} f = f.
$$

\end{proof}

Note that the probability measure $\theta$ above can be obtained by a \emph{stopping time}. More precisely, a function $\tau : \Omega \to \mathbb{N}^+$ is called a \emph{stopping time} if for any $n$ the set $\tau^{-1}\{n\}$ is measurable with respect to the $\sigma$-algebra generated by $x_1,\cdots,x_n$.  For any measurable set $A$,  define  
$$
\mu_\tau(A):=\PP(\{\x : x_{\tau(\x)} \in A\}).
$$
If the stopping time $\tau$ is almost surely finite, then $\mu_\tau$ is a probability measure. Moreover, for any $n$, $(\mu_\tau)^{*n} = \mu_{\tau_n}$
where $\tau_n$ is the iterated stopping time defined as $\tau_1 :=\tau$ and 
 $$
 \tau_{n+1} :=\tau_n+\tau\circ U^{\tau_n}. 
 $$
In particular, if $L \subseteq G$ is such that $\mu(L) > 0$, then one defines the stopping time 
$$\tau(\x) := \min \left \{ n > 0 : g_n \in L \right \}.$$
Then $\mu = \alpha + \beta$, where $\alpha = \mu_{\mid_L}$ and $\beta = \mu_{\mid_L^c}$, 
and  $\mu_\tau=\theta=\sum_{n=0} \beta^{*n}*\alpha$. 
The probability measure $\theta$ was introduced by Willis \cite{Willis}; transformations of random walks via stopping times are due to the first author and Kaimanovich, who show these transformations preserve the Poisson boundaries for random walks on discrete groups (for more details see \cite{Behrang-thesis} and \cite{Forghani}).

The next lemma compares the Furstenberg entropies of a $\mu$-boundary $(B, \lambda)$ with respect to the measure $\mu$ and 
the induced measure $\theta$. We denote by $h_\theta(\lambda)$ the Furstenberg entropy with respect to the probability measure $\theta$.

\begin{lemma} \label{L:finite-entro}
Let $(B,\lambda)$ be a $\mu$--boundary, and let $\theta$ be the induced measure as above, with corresponding stopping time $\tau$. Then 
$$
h_\theta(\lambda) = E(\tau) h (\lambda)
$$
and $E(\tau) =  \frac{1}{\| \alpha \|}.$
\end{lemma}

\begin{proof}
Since $\theta$ and $\mu$ have the same Poisson boundary, then $(B,\lambda)$ is also a $\theta$--boundary. Now, Theorem~\ref{theo : Furstenberg-entropy} implies that 
$$
\lim_n\frac{1}{n} \log \frac{dx_n \lambda}{d \lambda} (x^\lambda_\infty)= h(\lambda), \hspace{1cm}  \lim_n\frac{1}{n} \log \frac{dx_{\tau_n} \lambda}{d \lambda} (x^\lambda_\infty)= h_\theta(\lambda)
$$
for $\PP$--almost every sample path $(x_n)$, where we remark that $(x_{\tau_n})$ is a subsequence of $(x_n)$.  Note that $\lim_n \frac{\tau_n}{n} = \int \tau \ d\PP= \|\alpha\|^{-1}$. Therefore, 
$$
 h_\theta(\lambda)=\lim_n\frac{\tau_n}{n} \frac{1}{\tau_n} \log \frac{dx_{\tau_n} \lambda}{d \lambda} (x^\lambda_\infty)= \frac{1}{\|\alpha\|} h(\lambda).
$$
\end{proof}

\begin{lemma} \label{thm : finite moment tau}
Let $\g$ be a sub-additive gauge on the locally compact group $G$ and $|g|:=|g|_\g$.  If $L(\mu):= \int_G |g| \ d\mu(g)$ is finite, then for any stopping time $\tau$ with finite expectation $E(\tau)$, we have
$$
L(\theta) \leq E(\tau) L(\mu).
$$
 \end{lemma}
 
 \begin{proof}
 The proof is exactly the same as the one for discrete groups, see \cite{Forghani}. 
 \end{proof}

\begin{remark} \label{R:fake}
Note that when $G$ is a discrete group and $\mu$ has finite entropy, then $H(\mu_\tau) \leq E(\tau) H_1$, see \cite{Forghani}. However, in the case of locally compact groups, the same statement does not hold for differential entropies.  For instance, let $2\alpha$ be the uniform probability measure on $(0,1/2)$, 
as in Example~\ref{ex : fake}; then the differential entropy of $\mu$ with respect to the Haar measure is zero, but the differential entropy of
$\theta=\sum_{n=0}\beta^{*n}*\alpha$ with respect to the Haar measure is positive.
\end{remark}

The following theorem is mentioned in \cite{Kaimanovich-Zactions}. 

\begin{theorem}\label{thm : spread-out measure trick}
Let $\mu$ be a spread--out probability measure on $G$. Then there exists a probability measure $\theta$ on $G$ such that 
\begin{enumerate}
\item $\theta$ is absolutely continuous with respect to $m$, and $\frac{d \theta}{dm}$ is bounded; 
\item the Poisson boundaries of $(G,\mu)$ and $(G,\theta)$ are the same.
\end{enumerate}
Moreover, if $\mu$ has finite first moment with respect to a sub-additive, temperate gauge $\g$, then $\theta$ has also finite moment, and $H(\theta^{*n})$ is finite for any $n$. 
\end{theorem}

\begin{proof}
Suppose $\mu^{*n}$ is not singular with respect to the left Haar measure $m$. By the Lebesgue decomposition theorem,  we can write $\mu^{*n}=\mu_1+\mu_2$ such that $\mu_1$ is absolutely continuous with respect to $m$ and $\mu_2$ is singular with respect to $m$. Let choose $C>0$ such that  $\mu^{*n} (L)>0$, where
$$
L=\left\{ g \in G :  \frac{d \mu_1}{dm} (g) \leq C\right\}.
$$
For any measurable set $K$, define $\alpha(K):= \mu_1(K \cap L)$ and $\beta= (\mu_1-\alpha)+\mu_2$. Therefore, we have
$$
\mu^{*n}= \alpha +\beta.
$$
It is enough to define $\theta:= \sum_{i=0}^\infty \beta^{*i}*\alpha$. Note that the Poisson boundary of $\theta$ is the same as the one of $\mu^{*n}$, therefore, $\theta$ and $\mu$ have the same Poisson boundary. On the other hand, for any natural number $i$,  we have $\beta^{*i}*\alpha$ is absolutely continuous with respect to $m$ and 
$$
\left\| \frac{d ( \beta^{*i}*\alpha )}{d m} \right\|_\infty \leq C |  \beta |^i \Rightarrow \left\| \frac{d \theta}{ dm} \right\|_\infty \leq C.
$$
The last claims are a direct consequence of Corollary~\ref{cor : Hn finite} and Lemma \ref{thm : finite moment tau}.
\end{proof}

\subsubsection{The weak Shannon theorem for spread--out measures}

Combining Theorem~\ref{thm : bounded density}, Corollary~\ref{cor : Hn finite}, and Theorem~\ref{thm : spread-out measure trick} allows us to remove the 
bounded density condition.

\begin{theorem}
Let $G$ be a second countable  locally compact group, and let $\g$ be a temperate and sub-additive gauge. If  $\mu$ is spread-out with respect to the left Haar measure on $G$ and $\int |g|_\g \ d \mu$ is finite, then for $\PP$-almost every sample path $\x=(x_n)$
$$
\liminf_n - \frac1n \log \rho_n'(x_{\tau_n})=  E(\tau) h(\nu)
$$ 
where $\rho_n'$ is the density of $\theta^{*n}$, and $E(\tau)$ is the expectation of the stopping time $\tau$. 
\end{theorem}

\subsection{Entropy criterion and stopping times}\label{liminf stopping time}

\begin{lemma} \label{L:entro-stop}
Let $\mu$ be a probability measure on $G$, and let $(B, \lambda)$ be a $\mu$-boundary. 
Suppose that there exists a set $W \subseteq B$ such that for any $\gamma$ in $W$ and 
any $\epsilon > 0$ there exist measurable sets $A_{n, \epsilon}^\gamma \subseteq G$ with finite Haar measure such that
\begin{enumerate}
\item
$\liminf_n \PP^\gamma \left( x_n \in A_{n, \epsilon}^\gamma \right) > 0$
\item
$\limsup_n \frac{1}{n} \log m(A_{n, \epsilon}^\gamma) < \epsilon.$
\end{enumerate}
Let $\tau : \Omega \to \mathbb{R}$ a stopping time with finite expectation. 
Then there exist sets $\widetilde{A}_{n, \epsilon}^\gamma \subseteq G$ such that for any $\gamma \in W$
\begin{enumerate}
\item
$\liminf_n \PP^\gamma\left( x_{\tau_n} \in \widetilde{A}_{n, \epsilon}^\gamma \right) > 0$
\item
$\limsup_n \frac{1}{n} \log m(\widetilde{A}_{n, \epsilon}^\gamma) < \epsilon$.
\end{enumerate}
\end{lemma}

\begin{proof}
Let $\ell := \mathbb{E}(\tau)$ be the expectation of the stopping time. Then by the ergodic theorem $\lim_{n \to \infty} \frac{\tau_n(\x)}{n} = \ell$ 
for almost every $\x \in \Omega$. 
Let us fix $\gamma \in W$, and denote $\delta := \liminf_n \PP^\gamma \left( x_n \in A_{n, \epsilon}^\gamma \right) > 0$.
By Egorov's theorem, there exists a set $B \subseteq \Omega$ such that $\mathbb{P}^\gamma(B) \geq 1 - \frac{\delta}{2}$
and such that $\frac{\tau_n}{n} \to \ell$ uniformly on $B$. In particular, there exists $n_0$ such that for any $n \geq n_0$ 
and any $\x \in B$ we have 
$\tau_n(\x) \leq n( \ell + 1)$. Now, let us define 
$$\widetilde{A}_{n, \epsilon}^\gamma := \bigcup_{i = 1}^{n(\ell + 1)} A_{i, \epsilon}^\gamma.$$
Then we have 
$$\PP^\gamma\left( x_{\tau_n} \in \widetilde{A}_{n, \epsilon}^\gamma\right) \geq \PP^\gamma \left( x_{\tau_n} \in A_{\tau_n, \epsilon}^\gamma \textup{ and }\tau_n \leq n (\ell + 1) \right) \geq
 \PP^\gamma \left( x_{\tau_n} \in A_{\tau_n, \epsilon}^\gamma \right) - \frac{\delta}{2}$$
hence 
$$\liminf_n \PP^\gamma \left( x_{\tau_n} \in \widetilde{A}_{n, \epsilon}^\gamma \right) \geq \frac{\delta}{2}$$
which verifies (1). On the other hand, by (2) 
$$m(\widetilde{A}_{n, \epsilon}^\gamma) \leq \sum_{i=1}^{n(\ell+1)} m(A_{i, \epsilon}^\gamma) \leq n(\ell+1) e^{\epsilon n (\ell + 1)}$$
hence 
$$\limsup_n \frac{1}{n} \log m(\widetilde{A}_{n, \epsilon}^\gamma) \leq \epsilon(\ell+1)$$
which shows (2) for the stopping time.
\end{proof}

\section{Approximation criteria}

\subsection{Ray approximation}

We are now ready to prove the first geometric criterion to identify the Poisson boundary, the ray approximation criterion. 

\begin{theorem}[Ray approximation, version 1] \label{thm : ray1}
Let $G$ be a locally compact group, and let $\mu$ be a probability measure on $G$ with bounded density with respect 
to the left Haar measure on $G$. Suppose that the differential entropy $H_n$ is finite for all $n$, and let 
$(B,\lambda)$ be a $\mu$--boundary. 
Suppose there exists a measurable subset $W$ of $B$ with positive measure such that for any $b$ in $W$ there exists a sequence of 
uniformly temperate gauges $\g^n(b)$ such that for every $\epsilon>0$,
$$
\limsup_n\PP^b \left(  \x : \frac{1}{n}  | x_n |_{ \g^n( b)  } < \epsilon  \right)> 0.
$$
Then $(B,\lambda)$ is the Poisson boundary.
\end{theorem}

\begin{proof}
Let us define for any $b \in W$, $\epsilon > 0$ and $n \in \mathbb{N}$
$$
A^b_{n,{\epsilon}} := \left\{ g \in G : \frac{1}{n}  | g |_{ \g^n( b)  } < \epsilon \right\}. 
$$
By the assumption, $\limsup_n \PP^b \left( \x : x_n \in A^b_{n,\epsilon} \right) >0$ for all $\epsilon>0$. We also know that
$$
A^b_{n,\epsilon} \subset \g^n_{ [n\epsilon ] +1  } (b).
$$
Since $\g^n(b)$ is uniformly temperate, there exists $c>0$ such that $m(A_{n,\epsilon}) \leq e^{(1+n\epsilon) c}$.  Hence, $\limsup_n \frac{1}{n} \log m(A_{n,\epsilon})  \leq c\epsilon$. Therefore, by Theorem~\ref{thm : general An sets}, $(B,\lambda)$ is the Poisson boundary.
\end{proof}

Let us now extend this theorem to all measures with finite first moment. This will be done by combining the above Theorem with the stopping time trick; 
for this reason, we need to replace the $\limsup$ in the hypothesis by $\liminf$. 

\begin{theorem}[Ray approximation, version 2] \label{thm : ray2}
Let $\mu$ be a spread-out probability measure on a locally compact group $G$, with finite first moment w.r.t. a sub-additive, temperate gauge $| \cdot |_\g$ on $G$.
Let $(B,\lambda)$ be a $\mu$--boundary.  
Suppose there exists a measurable subset $W$ of $B$ with positive measure such that for any $b$ in $W$ there exists a sequence of 
uniformly temperate gauges $\g^n(b)$ such that for every $\epsilon>0$,
$$
\liminf_n\PP^b \left( \x : \frac{1}{n}  | x_n |_{ \g^n( b)  } < \epsilon  \right) > 0.
$$
Then $(B,\lambda)$ is the Poisson boundary.
\end{theorem}

Since almost sure convergence implies convergence in probability, we immediately obtain the following version.

\begin{corollary}[Ray approximation, version 3] \label{cor:ray-crit}
Let $\mu$ be a probability measure on a locally compact group $G$, with finite first moment w.r.t. a sub-additive, temperate gauge $| \cdot |_\g$ on $G$.
Let $(B,\lambda)$ be a $\mu$--boundary, and suppose that for almost every $b \in B$ there exists a sequence of 
uniformly temperate gauges $\g^n(b)$ such that 
$$\frac{1}{n} \left| x_n \right|_{\g^n(x_\infty^\lambda)} \to 0 \qquad \PP\textup{-a.s.}$$
Then $(B, \lambda)$ is the Poisson boundary. 
\end{corollary}

The most common way to apply the criterion is to show that the random walk lies at sublinear distance from 
a ``ray" connecting the starting point with the boundary point. More precisely, the following corollary immediately implies Theorem \ref{ray:intro}. 

\begin{corollary}[Ray approximation, version 4] \label{cor : ray-all}
Let $| \cdot |_\g$ be a sub-additive, temperate gauge on $G$, and suppose that $\mu$ has finite first moment w.r.t. $| \cdot |_\g$.
Assume that $(B,\lambda)$ be a $\mu$--boundary. If there exists a sequence of measurable functions $\pi_n: B \to G$ such that 
$$
\frac{1}{n} \left| \left( \pi_n(x^\lambda_\infty) \right)^{-1}x_n\right|_\g \to 0
$$
in $\PP$--probability, then $(B,\lambda)$ is the Poisson boundary.
\end{corollary}

\subsection{Strip approximation}

We also obtain the strip approximation criterion, which requires bilateral sample paths to be approximated by bilateral ``strips". 

\begin{theorem}[Strip approximation] \label{thm : Strip Approximation-stop}
Let $G$ be a locally compact group, and let $\mu$ be a spread-out probability measure on $G$. 
Let $(B_-,\lambda_-)$ be a $\check{\mu}$--boundary and $(B_+,\lambda_+)$ be a $\mu$--boundary. If there exists a sub-additive, temperate gauge $\g$ such that $\mu$ has finite first moment and a measurable $G$--equivariant map $S(b_-,b_+) \subset G$ such that  $\lambda_-\otimes \lambda_+\{(b_-,b_+) : e \in S(b_-,b_+) \}= q >0$ and 
 for $\lambda_-\otimes\lambda_+$-almost every $(b_-,b_+)$ in $B_-\times B_+$
$$
\frac{1}{n} \log^+ m\left( S(b_-,b_+) \cap \g_{|x_n|} \right) \to 0
$$
in $\PP$--probability, then $(B_+,\lambda_+)$ is the Poisson boundary.
\end{theorem}

The proof is almost exactly the same as for countable groups (as in \cite{Ka}, Section 6.4), using our new Theorem \ref{thm : bounded density} and the stopping time technique. For the sake of completeness, we will write it in the Appendix. 

\section{Applications} \label{S:app}

Let $(X, d)$ be a metric space, and consider a group $G$ which acts by isometries on $X$. Let us review some definitions on metric spaces: for more details, 
see \cite{CdH}.

Recall that a metric space is \emph{proper} if closed metric balls are compact.  A \emph{geodesic segment} between $x$ and $y$ is a map $\gamma : [0, d(x, y)] \to X$ such that $\gamma(0) = x$, $\gamma(d(x,y)) = y$ and 
$d(\gamma(s), \gamma(t)) = |s-t|$ for all $s, t \in [0, d(x,y)]$. A metric space is \emph{geodesic} if for any $x, y \in X$ there exists a 
geodesic segment between $x$ and $y$.

The metric on $X$ induces a gauge on $G$. Indeed, let us fix a base point $o \in X$. Then define $\g_k := \{ g \in G \ : \ d(o, go) \leq k \}$.
By the triangle inequality, such a gauge is sub-additive. We say that the action of $G$ on $X$ has \emph{bounded exponential growth} if there exists $C > 0$ such that 
$$m(\{ g \in G \ : \ d(o, go) \leq R \}) \leq C e^{CR} \qquad \textup{for all } R.$$
If the action has bounded exponential growth, then the gauge is temperate.
This condition is satisfied in many interesting examples as we will see below.

A space is \emph{coarsely connected} if there exists $C > 0$ such that for any $x, y \in X$ there exists a chain $x = x_0, x_1, \dots, x_k = y$ 
of points in $X$ such that $d(x_i, x_{i+1}) \leq C$ for any $i = 0, \dots, k-1$. In particular, a space which is path connected, such as a geodesic 
metric space, is coarsely connected. 

If $X$ is a non-empty, proper, coarsely connected metric space, and the action of $\textup{Isom}(X)$ on $X$ is cocompact, then 
the group $G = \textup{Isom}(X)$ is compactly generated and the action has bounded exponential growth.
This in particular holds if $X$ is a locally finite, connected graph, and the action of $\textup{Aut}(X)$ on $X$ is cobounded. 

More generally, if $G$ is a locally compact group with a \emph{geometric} (i.e. isometric, cobounded, metrically proper, and locally bounded\footnote{The action of $G$ on $X$ is \emph{metrically proper} if $S_{x, R} := \{ g \in G \ : \ d(x, gx) \leq R \}$
is relatively compact for any $x \in X$ and any $R \geq 0$. 
The action is \emph{locally bounded} if for any $g \in G$ and any bounded set $B \subseteq X$, there exists a neighbourhood 
$V$ of $g$ in $G$ such that $VB$ is bounded.} ) action 
on a coarsely connected metric space $X$, then $G$ is compactly generated and the action has bounded exponential growth.

 Now, let us endow $G = \textup{Isom}(X)$ with the compact open topology. If $X$ is proper, then the group $\textup{Isom}(X)$ is a locally compact, second countable, Hausdorff topological group  (see e.g. \cite{CdH}, Lemma 5.B.4), hence it carries a left Haar measure $m$. 
Recall also that in general a closed subgroup of a locally compact group is locally compact. 

Let us call a measure $\mu$ on a locally compact group $G$ \emph{weakly spread-out} if $\mu$ is spread-out with respect 
to the Haar measure on the closed subgroup generated by the support of $\mu$. 

\subsection{Automorphisms of trees}

Let $\mathbb{T}$ be a locally finite, infinite simplicial tree. The automorphism group $G$ of $\mathbb{T}$ is a totally disconnected, locally compact group. 
As shown by Cartwright-Soardi  \cite{CS}, if the group generated by the support of $\mu$ is not amenable, 
then the random walk converges almost surely to the space of ends of $\mathbb{T}$.

Let us remark that an interesting particular case is the group $G = \textup{PGL}(2, \mathbb{K})$ where 
$\mathbb{K}$ is a non-archimedean local field (see also Section \ref{S:affine}). 

\begin{theorem}
Let $\mu$ be a regular Borel probability measure on the group $G$ of automorphisms of a locally finite tree $\mathbb{T}$. 
Assume that $\mu$ has finite first moment for the simplicial metric on $\mathbb{T}$ and that the support of $\mu$ is not contained in any amenable subgroup of $G$. 
Then the space of ends of $\mathbb{T}$ with the hitting measure is the Poisson boundary of $(G, \mu)$. 
\end{theorem}

\begin{proof}
By \cite{CS}, the space of ends of $\mathbb{T}$ is a $\mu$-boundary. 
Define the strip $S(b_-, b_+)$ as the $K$-neighborhood of the geodesic connecting $b_-$ and $b_+$. By choosing $K$ large enough, we can make 
sure that $\lambda_- \otimes \lambda_+ (e \in S(b_-, b_+) ) > 0$. The conclusion then follows from Theorem \ref{thm : Strip Approximation-stop}. 
\end{proof}

At least for the full automorphism group, the result also follows from \cite{Bader-Shalom}, who also obtain the Poisson boundary for groups acting strongly transitively on locally finite, 
affine, thick buildings.

\subsection{Groups of isometries of Gromov hyperbolic spaces}

Let $(X, d)$ be a metric space. The \emph{Gromov product} of two points of $x$ and $y$ with respect to $z$ in $X$ is defined as
$$
(x,y)_z := \frac{d(x,z)+d(y,z)-d(x,y)}{2}.
$$
The metric space $(X,d)$ is \emph{$\delta$-hyperbolic} if for any points $x,y,z,w$, 
$$
(x,y)_w \geq \min\{ (x,z)_w, (y,z)_w \} - \delta.
$$

A metric space $(X, d)$ is \emph{quasiconvex} if there exists a constant $C > 1 $ such that for any two points $x, y$ in $X$,
there is a sequence $(x_i)_{i = 1}^k$ of points of $X$ such that $x_1 = x$, $x_k = y$ and for any $1 \leq i < j \leq k$ we have 
$$j -i - C \leq d(x_i, x_j) \leq j-i + C.$$
In particular, a space which is \emph{geodesic}, i.e. any two points can be joined by a geodesic, is quasiconvex. 

Let $G$ be a group of isometries of $(X, d)$. Given a probability measure $\mu$ on $G$, let $G_\mu$ be the smallest closed 
subgroup of $G$ of full measure.
The measure $\mu$ is \emph{non-elementary} if the group $G_\mu$ is unbounded, and the action of $G_\mu$ on $\partial X$
has no finite orbits. 

\begin{proposition}[\cite{BQ}, Proposition 3.2]
Let $G$ be a group of isometries of a proper, quasiconvex Gromov hyperbolic space $X$, let $o \in X$ be a base point, and let $\mu$ be a non-elementary probability measure
on $G$. Then for almost every sample path $(x_n)$ the sequence $x_n o$ converges to a point in $\partial X$, and the space $(\partial X, \nu)$, where $\nu$ is the hitting measure, is a $\mu$-boundary.
\end{proposition}

We therefore get the following identification of the Poisson boundary. 

\begin{theorem} \label{T:hyperbolic}
Let $G$ be a closed group of isometries of a proper, quasiconvex Gromov hyperbolic space $X$ with bounded exponential growth, and let $\mu$ be a non-elementary, weakly spread-out probability measure on $G$ with finite first moment. Then the space $(\partial X, \nu)$ is a model for the Poisson boundary of $(G, \mu)$. 
\end{theorem}

\begin{proof}
By (\cite{Ka}, Theorem 7.2), almost every sequence $(x_n)$ is regular, that is there exists a geodesic ray $\gamma : [0, \infty) \to X$ 
such that $\frac{d(x_n o, \gamma(An )) }{n} \to 0$. Then for any boundary point $b \in \partial X$ and any $n$ one defines the 
gauge $\mathcal{G}^n(b)$ as the metric gauge around $\gamma(An)$, i.e.
$$(\mathcal{G}^n(b))_k := \{ g \in G \ : \ d(\gamma(An), x_n o)\leq k \}.$$
This gauge is temperate since the action has bounded exponential growth, hence the theorem follows from the ray approximation criterion 
(Theorem \ref{cor:ray-crit}). 
\end{proof}

Let us remark that one can also obtain the same result by applying the strip approximation, by choosing the following strips. 
For each pair $(\xi, \eta) \in \partial X \times \partial X$ with $\xi \neq \eta$, define as $[\xi, \eta]$ the union of all bi-infinite geodesics 
which connect $\xi$ and $\eta$. Then, define $S(\xi, \eta)$ to be the set 
$$S(\xi, \eta) := \{ g \in G \ : \ d(g, [\xi, \eta]) \leq C \}.$$
If $C$ is large enough, then $\mathbb{P}(e \in S(\xi, \eta)) > 0$. Hence, the result follows from Theorem \ref{thm : Strip Approximation-stop}.

\subsection{Lie groups}

Let $G$ be a connected, semi-simple Lie group with finite center. The Poisson boundary of such groups 
has been characterized by Furstenberg \cite{Furstenberg-semi}. Let us see how ray approximation yields a proof for any closed subgroup 
of a semisimple Lie group (which need not be semisimple). 
We follow the notation from \cite{Ka}.

Let $S = G/K$ be the quotient of $G$ by its maximal compact subgroup $K$. Let us fix the base point $o = K \in S$, 
and consider the visual compactification $\partial S$ of $S$.
Let $\mathfrak{U}^+$ be a dominant Weyl chamber. Then every point of $S$ can be written as $x = k \exp(a) o$ where $k \in K$ and $a \in \overline{\mathfrak{U}^+}$. We will denote as $r(x) = a$ the \emph{radial part} of $x$. 
Then for each $a \in \overline{\mathfrak{U}^+}$, consider the set $\partial S_a$ of limit points of all geodesics of the form $\gamma(t) := g \exp(t a) o$.

Note that if $a \in \mathfrak{U}^+$ belongs to the interior of the chamber, then $\partial S_a$ is isomorphic to the \emph{Furstenberg boundary} $B = G/P$ where $P = MAN$ is a minimal parabolic subgroup. 

\begin{theorem} \label{T:Lie}
Let $\mu$ be a spread-out measure with finite first moment on a closed subgroup of a connected, semi-simple Lie group $G$ with finite center. 
Then there exists $a$ such that 
$$\lim_{n \to \infty} \frac{r(x_n)}{n} = a \in \overline{\mathfrak{U}^+}$$
almost surely. 
Moreover, if $a = 0$, then the Poisson boundary of $(G, \mu)$ is trivial; if $a \neq 0$, then almost every sample path $(x_n)$ converges 
to the visual boundary $\partial S_a$, and $\partial S_a$ with the hitting measure is a model for the Poisson boundary of $(G, \mu)$. 
\end{theorem}


\begin{proof}
By \cite{Ka89}, for almost every $\x = (x_n)$ there exists a geodesic ray $\gamma : [0, \infty) \to X$ such that 
$d(x_n, \gamma(An)) = o(n)$. We then apply Theorem \ref{cor:ray-crit}.
\end{proof}

\subsection{Groups of isometries of CAT(0) or nonpositively curved spaces}

Another large class of groups we can apply our criterion to is groups of isometries of nonpositively curved spaces in the sense of Busemann, 
which in particular includes the case of CAT(0) spaces. 

Let us review the relevant definitions, following \cite{KM}. Given two points $x, y$ in a metric space $(X, d)$, $z$ is a \emph{midpoint} of $x, y$ if $d(x, z) = d(z,y) = d(x, y)/2$. A metric space is \emph{convex} if every pair of points $x, y$ has a midpoint in $X$, which will be denoted as $m_{xy}$.
Moreover, a space is \emph{uniformly convex} if there is a strictly decreasing, continuous function $g : [0, 1] \to \mathbb{R}$
with $g(0) = 1$ and for any $x, y , w \in X$
$$\frac{d(z, w)}{R} \leq g\left( \frac{d(x, y)}{2R} \right) $$
where $z$ is the midpoint of $x$ and $y$, and $R = \max\{d(x,w), d(y, w) \}$. 
Moreover, a convex metric space is \emph{Busemann nonpositively curved} if for any $x, y, z \in X$ one has 
$$d(m_{xz}, m_{yz}) \leq \frac{d(x, y)}{2}.$$

Let $(X, d)$ be a uniformly convex, complete metric space which satisfies the Busemann non-positive curvature condition, 
and let $G$ be a group of isometries of $X$. In particular, CAT(0) spaces and uniformly convex Banach spaces satisfy this condition.

In this setting, the ray approximation criterion is provided by Karlsson-Margulis:
\begin{theorem}[\cite{KM}] \label{T:KM}
Let $(X, d)$ be a uniformly convex, complete metric space which satisfies the Busemann non-positive curvature condition, 
and let $G$ be a group of isometries of $X$.
Let $\mu$ be a probability measure on $G$ with finite first moment on $X$, and let $o \in X$ be a base point. 
Then there exists $A \geq 0$ such that for almost every sample path $\x = (x_n)$ there exists a unique geodesic ray $\gamma : [0, \infty) \to X$ such that $\gamma(0) = o$ and 
$$\lim_{n \to \infty} \frac{1}{n} d(\gamma(An), x_n o) = 0.$$
\end{theorem}

Now, the space $X$ is endowed with its \emph{visual boundary} $X(\infty)$, defined as the set of asymptote classes of geodesic rays starting from $o$. 
Applying the ray approximation criterion yields the following identification of the Poisson boundary: 

\begin{theorem}
Let $G$ be a locally compact group of isometries of a uniformly convex, complete metric space $(X, d)$ which satisfies the Busemann non-positive 
curvature condition. Let $\mu$ be a weakly spread-out measure on $G$ with finite first moment on $X$, and suppose that the action has bounded exponential growth. 
Let $A \geq 0$ denote the drift of the random walk. Then: 
\begin{enumerate}
\item if $A = 0$, then the Poisson boundary of $(G, \mu)$ is trivial; 
\item if $A > 0$, then the random walk converges almost surely to the visual boundary $X(\infty)$, and the visual boundary $(X(\infty), \nu)$ with the 
hitting measure is the Poisson boundary of $(G, \mu)$. 
\end{enumerate}
\end{theorem}

\begin{proof}
If $A = 0$, then one can just take as $(B, \lambda)$ the trivial one-point compactification, which is always a $\mu$-boundary. 
Then the gauge $\mathcal{G}_k := \{ g \ : \ d(o, go) \leq k \}$
is temperate, hence by Corollary \ref{cor:ray-crit} the Poisson boundary is trivial.

If $A > 0$, then the random walk converges almost surely to $X(\infty)$, and the visual boundary is a $\mu$-boundary. 
For each $b \in X(\infty)$, let us pick the geodesic ray $\gamma$ which joins $o$ and $b$, and let us define the gauge $\g^n(b)$ 
as 
$$(\mathcal{G}^n(b))_k := \{ g \ : \ d(go, \gamma(An)) \leq k \}.$$ 
Since the action has bounded exponential growth, all these gauges are uniformly temperate. Moreover, by Theorem \ref{T:KM}, 
the hypotheses of Corollary \ref{cor:ray-crit} are satisfied, hence $(X(\infty), \nu)$ is the Poisson boundary.
\end{proof}

\subsection{Affine groups} \label{S:affine}

The \emph{real affine group} 
$$\textup{Aff}(\mathbb{R}) := \{ x \mapsto ax + b, a > 0, b \in \mathbb{R}\}$$ 
has been 
considered by Azencott \cite{Azencott}. In that case, the finite first moment condition 
is that $\int \log^+(|a| + |b|) \ d\mu < + \infty$, and one defines 
$$\phi(\mu) := \int \log a \ d\mu.$$
The characterization of the Poisson boundary then follows by Theorem \ref{T:hyperbolic}. Indeed, let us note that $\textup{Aff}(\mathbb{R}) < GL(2,\mathbb{R})$ acts by M\"obius transformations on the upper half plane $\mathbb{H}^2$, and these are isometries for the hyperbolic metric. 

Now, if $\phi(\mu) = 0$, then the rate of escape $A = 0$, hence the Poisson boundary is trivial. 
Otherwise, the Poisson boundary is given by $\partial \mathbb{H}^2$ with the hitting measure. 
Now, if  $\phi(\mu) > 0$, all sample paths converge to $\{ \infty \}$, hence the hitting measure is the point mass at $\infty$, 
thus the Poisson boundary is trivial.
If $\phi(\mu) < 0$, then almost every sample path converges to $\partial \mathbb{H}^2 \setminus \{ \infty \} = \mathbb{R}$, hence 
the Poisson boundary is $(\mathbb{R}, \nu)$. This recovers the result from \cite{Azencott}.

\bigskip

More generally, if $\mathbb{K}$ is a field, then the \emph{affine group} over $\mathbb{K}$ is the group 
$$\textup{Aff}(\mathbb{K}) := \left\{ \left( \begin{array}{ll} a & b \\ 0 & 1 \end{array} \right) \ : \ a \in \mathbb{K}^\star, b \in \mathbb{K} \right\}.$$
A case of interest is when $\mathbb{K}$ is a non-archimedean local field. 
Following \cite{Cartwright-Kaimanovich-Woess}, one notes that the Bruhat-Tits building of $\mathbb{K}$ is a 
homogeneous, locally finite tree of valence $q + 1$, where $q$ is the cardinality of the residue field. Hence, $\textup{Aff}(\mathbb{K})$ 
is a subgroup of the automorphism group of a tree. 

In general, let $\mathbb{T}$ be a homogeneous, locally finite tree, and let us fix an end $\omega$ of $\mathbb{T}$ and a base point $o \in \mathbb{T}$. 
Then, one defines $\textup{Aff}(\mathbb{T})$ as the group of automorphisms of $\mathbb{T}$ which fix $\omega$.
Let $h : \mathbb{T} \to \mathbb{Z}$ be the \emph{Busemann function} associated to $\omega$; this is defined as 
$$h(x) := d(x, c) - d(o, c)$$
where $c$ is the closest point projection of $x$ onto the geodesic ray joining $o$ and $\omega$. 
Now we define the homomorphism $\phi: \textup{Aff}(\mathbb{T}) \to \mathbb{Z}$ as $\phi(g) := h(g o)$. We will denote as $\phi(\mu)$ 
the expected value of $\phi$. 
 
\begin{theorem}
Let $\mu$ be a weakly spread-out probability measure on $\textup{Aff}(\mathbb{T})$ with finite first moment, and suppose that the closed subgroup generated 
by the support of $\mu$ is non-exceptional (i.e., its limit set is uncountable). 
Then if $\phi(\mu) \leq 0$, the Poisson boundary is trivial; 
if $\phi(\mu) > 0$, then the Poisson boundary is $(\partial \mathbb{T}, \nu)$, where $\nu$ is the hitting measure on $\partial \mathbb{T}$. 
\end{theorem}

\begin{proof}
Ray approximation holds, by \cite{Cartwright-Kaimanovich-Woess}, equation (3.10). 
\end{proof}

\subsection{Automorphisms of Diestel-Leader graphs} \label{S:DL}

Diestel-Leader graphs \cite{DL} have been introduced as examples of vertex-transitive graphs which are not quasi-isometric 
to any Cayley graph, as an answer to a question of Woess (see \cite{Woess-treebolic}).

Fix two integers $p, q \geq 2$, and consider the regular trees $\mathbb{T}_p$ and $\mathbb{T}_q$ of degrees $p$ and $q$. 
Fix two ends $\omega_1 \in \partial \mathbb{T}_p$ and $\omega_2 \in \partial \mathbb{T}_q$. The choice of an end $\omega_i$ 
defines a \emph{Busemann function} $h_i : \mathbb{T}_i \to \mathbb{Z}$. 

The Diestel-Leader graph $DL_{p,q}$ is the graph with vertex set
$$DL_{p, q} := \{ (z_1, z_2) \in \mathbb{T}_p \times \mathbb{T}_q \ : \ h_1(z_1) + h_2(z_2) = 0 \}.$$ 
and edges $(z_1, z_2) \sim (y_1, y_2)$ if $z_1 \sim y_1$ and $z_2 \sim y_2$.
Let $d$ denote the graph metric on $DL_{p, q}$, and pick a base point $o = (o_1, o_2)$. 

If $p \neq q$, $DL_{p, q}$ is not quasi-isometric to any Cayley graph \cite{Eskin-F-W}, while $DL_{p, p}$ is a Cayley graph 
of the lamplighter group $\mathbb{Z}_p \wr \mathbb{Z}$ \cite{Woess-trees}. 

Consider the group 
$$G_{p,q} = \{ (g^1, g^2) \in \textup{Aff}(\mathbb{T}_p) \times \textup{Aff}(\mathbb{T}_q) \ : \ h_1(g^1 o_1) + h_2(g^2 o_2) = 0 \}$$
which acts by isometries on $DL_{p, q}$. Note that when $p\not = q$, then $G_{p,q}$ is the full automorphism group of $DL_{p,q}$, 
while $G_{p,p}$ is a subgroup of index 2 in the full automorphism group of $DL_{p,p}$, see \cite{Woess-trees}.
Define the \emph{vertical drift} as 
$$V := \int_G h_1(g^1 o_1) \ d\mu(g),$$
which is finite if $\mu$ has finite first moment.

\begin{theorem} \label{T:DL}
Let $G$ be a closed subgroup of the group $G_{p,q}$ of automorphisms of the Diestel-Leader graph $DL_{p, q}$, and let $\mu$ be a spread-out probability measure on $G$
with finite first moment. Then: 
\begin{enumerate}
\item if $V = 0$, then the Poisson boundary of $(G, \mu)$ is trivial; 
\item if $V > 0$, then the Poisson boundary of $(G, \mu)$ is $\partial \mathbb{T}_p \times \{ \omega_2 \}$ with the hitting measure; 
\item  if $V < 0$, then the Poisson boundary of $(G, \mu)$ is $\{\omega_1 \} \times \partial \mathbb{T}_q$ with the hitting measure.
\end{enumerate}

\end{theorem}

\begin{proof}
By Lemma \ref{L:autX}, the gauge induced on $G$ by the metric $d$ is temperate.
Suppose $V > 0$. Let $(x_n)$ be a sample path, and denote as $x_n = (x_n^1, x_n^2)$ the two components, with $x_n^1 \in \textup{Aff}(\mathbb{T}_p)$ and $x_n^2 \in \textup{Aff}(\mathbb{T}_q)$. Since trees are $\delta$-hyperbolic, by (\cite{Ka}, Theorem 7.2) for almost every $(x_n)$ there exists $\xi_1 \in \partial \mathbb{T}_p$ such that $x_n^1 o_1 \to \xi_1$, and moreover there exists a geodesic ray $\gamma_1 : [0, \infty) \to \mathbb{T}_p$ between $o_1$ and $\xi_1$ such that 
$$\frac{d_1(x_n^1 o_1, \gamma_1(Vn))}{n} \to 0.$$ 
Similarly, almost surely one also has a geodesic ray $\gamma_2$ between $o_2$ and $\omega_2$ in $\mathbb{T}_q$ so that 
$\frac{d_2(x_n^2 o_2, \gamma_2(Vn))}{n} \to 0$.
Recall that the distance on $DL_{p, q}$ is given by (\cite{Bertacchi})
$$d((z_1, z_2), (y_1, y_2)) = d_1(z_1, y_1) + d_2(z_2, y_2) - |h_1(y_1) - h_1(z_1)|.$$
This means that, by choosing $\gamma_n := (\gamma_1(\lfloor Vn \rfloor), \gamma_2(\lfloor Vn \rfloor))$, we have 
$$\frac{d(x_n o, \gamma_n)}{n} \to 0.$$
The claim then follows by ray approximation (Corollary \ref{cor : ray-all}). 
The case $V < 0$ is completely analogous, reversing the roles of $\mathbb{T}_p$ and $\mathbb{T}_q$.
If $V = 0$, then a.e. sample path sublinearly tracks the point $o$, hence the Poisson boundary is trivial again by Corollary \ref{cor : ray-all}.
\end{proof}

\subsection{Horocylic products} \label{S:horo}

A generalization of the construction of Diestel-Leader graphs is the notion of \emph{horocylic product} of metric spaces,
due to Woess (see \cite{Woess-treebolic} for a survey). 

Let $X_1, \dots, X_k$ be $k$ proper, geodesic, $\delta$-hyperbolic spaces, and pick for each $i$ a boundary point $\omega_i \in \partial X_i$. 
Then denote as $h_i : X_i \to \mathbb{R}$ the Busemann function associated to $\omega_i$. 
The \emph{affine group} of $(X_i, \omega_i)$ is defined as 
$$\textup{Aff}(X_i) := \{ g \in \textup{Isom}(X_i) \ : \ g\omega_i = \omega_i \}.$$
We define the \emph{horocylic product} as the 
space 
$$ \mathfrak{X} = \prod_{\mathfrak{h}} X_i  := \{ (z_1, z_2, \dots, z_k) \in X_1 \times \dots \times X_k \ : \ \sum_{i=1}^k h_i(z_i) = 0 \}$$
As explained in Woess \cite{Woess-treebolic}, horocylic products include the following cases: 
\begin{enumerate}
\item
if $X_1 = \mathbb{T}_p$, $X_2 = \mathbb{T}_q$, you get the Diestel-Leader graph $DL_{p, q}$. 
\item
If $X_1 = \mathbb{H}(p)$ is the rescaled hyperbolic plane with curvature $- p^2$ and $X_2 = \mathbb{T}_q$, one gets the \emph{treebolic spaces}, on which Baumslag-Solitar groups 
act cocompactly. 
\item
if $X_1 = \mathbb{H}(p)$ and $X_2 = \mathbb{H}(q)$, then the horocylic product is the \emph{Sol-group}  
$$S(p ,q) := \left\{ \left( \begin{array}{lll} e^{pc} & a & 0 \\ 0 & 1 & 0 \\ 0 & b & e^{-qc} \end{array}\right) \ : \ a,b,c \in \mathbb{R}  \right\}.$$ 
This is a connected, solvable, Lie group (in particular, it is not semisimple).
\end{enumerate}
Consider the action on $\mathfrak{X}$ of the group 
$$\textup{Aut}(\mathfrak{X}) := \{ (g_1, \dots, g_k) \in \textup{Aff}(X_1) \times \dots \times \textup{Aff}(X_k) \ : \ \sum_{i=1}^k h_i(g_i o_i) = 0 \}.$$
For instance, when $X_i = \mathbb{T}_p$ for all $i$, then $\textup{Aut}(\mathfrak{X})$ is a ``higher rank lamplighter group" as studied 
in \cite{B-N-Woess}. 
We define a metric $d$ on $\mathfrak{X}$ to be \emph{compatible} if $d((z_1, \dots,z_k), (y_1, \dots, y_k)) \leq \sum_{i =1}^k d_i(z_i, y_i)$.
An example is given by the metric 
$$d((z_1, \dots, z_k), (y_1, \dots, y_k)) = \sum_{i = 1}^k d_i(z_i, y_i) - \sum_{i =1}^{k-1} |h_i(z_i) - h_i(y_i)|,$$
but also the usual metrics on treebolic spaces and on Sol are compatible (\cite{Woess-trees}, Lemmas 2.11 and 2.18).
For each $i$, let us define \emph{the drift in direction} $i$ as 
$$V_i := \int_G h_i( g_i o) \ d\mu(g).$$
The same proof as in Theorem \ref{T:DL}, but with $k$ components, yields:

\begin{theorem} \label{T:horocyclic}
Let $G$ be a closed group of isometries of a horocylic product $\mathfrak{X}$ of proper, $\delta$-hyperbolic metric spaces 
with a compatible metric, suppose that the action has bounded exponential growth, 
and let $\mu$ be a weakly spread-out measure on $G$ with finite first moment. 
Then the Poisson boundary of $(G, \mu)$ is 
$$\prod_{V_i > 0} \partial X_i \times \prod_{V_i < 0} \{ \omega_i \}$$ 
with the corresponding hitting measure, where $\partial X_i$ is the Gromov boundary of $X_i$.
\end{theorem}

The Theorem extends the result of \cite{B-N-Woess} on the horocylic products of $k$ trees, and even in that case 
the measure $\mu$ is no longer constrained to be invariant by a point stabilizer.
Moreover, it also yields the Poisson boundary of groups of isometries of treebolic spaces (including the Baumslag-Solitar groups $BS(1, p)$) and Sol-groups $S(p, q)$.

\section{Appendix}

\subsection{Proof of the Entropy criterion (Theorem~\ref{thm : general An sets}) } 
\begin{proof}
For a fixed $\gamma \in W$ and $\epsilon > 0$, let us define the set
$$
B^\gamma_n=\left\{g\in A^\gamma_{n,\epsilon} : \rho_n(g) \frac{d g \lambda}{d \lambda}( \gamma ) \leq e^{-4n\epsilon} \right\}.
$$
We have 
$$
\PP^\gamma  \{ \x : x_n \in B_n^\gamma\}=\int _{B^\gamma_n} \rho_n(g) \frac{d g \lambda}{d \lambda} (\gamma) dm(g) \leq m(B^\gamma_n) e^{-4n\epsilon} \leq m(A^\gamma_{n,\epsilon}) e^{-4n\epsilon} .
$$
We have $\limsup_n\frac1n \log m(A^\gamma_{n,\epsilon})<\epsilon$, therefore,  there  exists $N$ such that 
$$
\sup_{n\geq N}\frac{1}{n}\log m(A^\gamma_{n,\epsilon})<2\epsilon,
$$
we can conclude that for any $n \geq N$ 
$$
\PP^\gamma \{ \x : x_n \in B_n^\gamma \} \leq e^{-2n \epsilon}.
$$
Now, let us define
$$
C^\gamma_n:= A^\gamma_{n,\epsilon} \setminus B^\gamma_{n}  = \left\{g\in A^\gamma_{n,\epsilon} : \rho_n(g) \frac{d g \lambda}{d \lambda} (\gamma) \geq  e^{-4n\epsilon}    \right\}.
$$   
Combining this with hypothesis (1), we get
$$\limsup_n\PP^\gamma \{\x : x_n\in C^\gamma_n \}  =  \limsup_n\PP^\gamma \{\x : x_n\in A^\gamma_{n, \epsilon} \}  > 0$$
hence if we set $\Omega_\gamma := \{ \x \in \Omega : x_\infty^\lambda = \gamma \textup{ and }x_n \in C_n^\gamma \mbox{ infinitely often }\} $ then 
$$\PP^\gamma(\Omega_\gamma) = \PP^\gamma \{ \x : x_n \in C_n^\gamma \mbox{ infinitely often }\} > 0.$$
Now, if a sample path $\x = (x_n)$ is such that $x_n$ belongs to $C^\gamma_n$, then by definition
$$
\log \rho_n(x_n) +\log \frac{d x_n \lambda}{d \lambda} (\gamma) \geq -4n\epsilon.
$$
So if $\x \in \Omega_\gamma$ then 
$$\liminf_n \left(  -\frac{1}{n}\log \rho_n(x_n) \right) + \liminf_n \left( - \frac{1}{n} \log \frac{d x_n \lambda}{d \lambda} (\gamma) \right)  \leq 4 \epsilon.$$ 
Now, by Theorem~\ref{thm : bounded density} and Theorem \ref{theo : Furstenberg-entropy}, there exists a set $S \subseteq \Omega$ 
with $\PP(S) = 1$ such that for $\x \in S$ 
$$\liminf_n \left(  -\frac{1}{n}\log \rho_n(x_n) \right) + \liminf_n \left( - \frac{1}{n} \log \frac{d x_n \lambda}{d \lambda} (x_\infty^\lambda) \right)  = h(\nu) - h(\lambda).$$
Since $\PP(\bigcup_{\gamma \in W} \Omega_\gamma) = \int_{\gamma \in W} \mathbb{P}^\gamma(\Omega_\gamma) \ d\lambda(\gamma) > 0$, then 
there exists $\gamma \in W$ and $\x \in S \cap \Omega_\gamma$, thus by comparing the two previous equations and noting that $\gamma = x_\infty^\lambda$ we get 
$$ h(\nu)-h(\lambda) \leq 4\epsilon, $$ 
and by choosing an arbitrary $\epsilon$ we obtain $h( \nu )= h(\lambda)$.
\end{proof}

\subsection{Proof of the Strip Criterion (Theorem~\ref{thm : Strip Approximation-stop})}
\begin{proof}
Since $\mu$ is spread-out, without loss of generality we can assume that $\mu$ is not singular with respect to the Haar measure. Therefore, there exist finite measures $\mu_1$ and $\mu_2$ such that $\mu=\mu_1+\mu_2$ and $\mu_1$ is absolutely continuous w.r.t the Haar measure. Let $\rho_1:= \frac{d \mu_1}{dm}$.

Let us denote as $\xi_+ : \Omega \to B_+$ and $\xi_- : \Omega \to B_-$ the boundary maps.
Then we have  
$$
S(\xi_+(\x), \xi_-(\x))=\bigcup_{j=1}^\infty \left(S(\xi_+(\x), \xi_-(\x))\cap\{g \in G : \rho_1(g) \leq j\}\right);
$$
since $\overline{\PP}(e \in S(\xi_+(\x), \xi_-(\x)))= q >0$, there exists $j>0$ such that 
$$
\overline{\PP}(e \in S(\xi_+(\x), \xi_-(\x)) \textup{ and }\rho_1 (g_1) \leq j) = q_1 > 0.
$$
Let $L=\{g \in G : \rho_1(g) \leq j\}$ and define $C_L := \{ (g_n) \in G^\mathbb{Z} \ : \ g_1 \in L \} \subseteq G^\mathbb{Z}$ the set of sequences whose first entry lies in $L$. If $T$ denotes the left shift on the space of increments, then the 
return time is $\tau(\x) := \inf \{ k \geq 1 \ : \ T^k(\x) \in C_L \}$ and we define $\mathfrak{T}(\x) := T^{\tau(\x)}(\x)$ the induced transformation.
Then the composition is given by $\mathfrak{T}^n(\x) = T^{\tau_n(\x)}(\x)$ where 
$$\tau_n(\x) = \sum_{k = 0}^{n-1} \tau(\mathfrak{T}^k(\x)).$$
Then $\mathfrak{T}\vert_{C_L}$ is measure preserving, hence for any $n$
$$\overline{\PP}\vert_{C_L}(x_{\tau_n} \in S(\xi_+, \xi_-)) = \overline{\PP}\vert_{C_L}(e \in S(\xi_+, \xi_-)) $$
which implies 
$$\overline{\PP}(x_{\tau_n} \in S(\xi_+, \xi_-)) \geq \overline{\PP}(e \in S(\xi_+, \xi_-) \textup{ and }g_1 \in L) = q_1 > 0 $$
hence 
\begin{equation} \label{E:S1}
\int \widetilde{p}_n^{b^+}(S(b_-, b_+)) \ d\lambda_-(b_-) d\lambda_+(b_+) = p > 0
\end{equation}
where $\widetilde{p}_n^{b^+}$ is the $n^{th}$-step distribution of the conditional random walk with boundary point $b^+$.
Now, we define
$$K_{n} := \min \{ k \geq 1 \ : \ \theta^{*n}(\mathcal{G}_{k}) \geq 1 - p/2 \}$$
and then 
\begin{equation} \label{E:S2}
\int \widetilde{p}_n^{b^+}(\mathcal{G}_{K_n})\ d\lambda_+(b_+) = \mathbb{P}(|x_{\tau_n}| \leq K_n) \geq 1 -p/2.
\end{equation}
Now, by \eqref{E:S1}
$$\int \widetilde{p}_n^{b^+}[S(b_-, b_+) \cap \mathcal{G}_{K_n}]\ d\lambda_-(b_-) d \lambda_+(b_+) \geq p/2$$
hence 
\begin{equation} \label{E:p5}
\lambda_- \otimes \lambda_+[(b_-, b_+) \ : \ \widetilde{p}_n^{b^+}[S(b_-, b_+) \cap \mathcal{G}_{K_{n}}] \geq p/4] \geq p/4.
\end{equation}
On the other hand, by minimality in the definition of $K_n$
$$\mathbb{P}(|x_{\tau_n}| < K_n) < 1 -p/2$$
hence 
\begin{equation} \label{E:tau-Kn}
\mathbb{P}(|x_{\tau_n}| \geq K_n) > \frac{p}{2}.
\end{equation}
Now, pick a generic pair $(b_-, b_+) \in B_- \times B_+$, and fix $\epsilon > 0$. Since $\frac{1}{n} \log^+ m( S(b_-, b_+) \cap \mathcal{G}_{|x_n|} ) \to 0$ in probability, 
then for any $\epsilon > 0$ and any $L > 0$
$$\mathbb{P}\left( m( S(b_-, b_+) \cap \mathcal{G}_{|x_{Ln}|} ) \leq e^{\epsilon n} \right) \to 1$$
as $n \to \infty$. Now, since the random walk has finite first moment, $\frac{|x_n|}{n} \to \ell$ almost surely, hence also $\frac{|x_{\tau_n}|}{\tau_n} \to \ell$ almost surely. 
Moreover, since the stopping time has finite first moment, $\frac{|x_n|}{|x_{\tau_n}|} \to \ell'$ almost surely, hence there exists $\ell'' > 0$  such that 
$$\PP(|x_{\tau_n}| \leq |x_{\ell'' n}|) \to 1$$
hence for a.e. pair $(b_-, b_+)$ one has 
$$\limsup_n \frac{1}{n} \log^+ m(S(b_-, b_+) \cap \g_{|x_{\tau_n}|}) = 0$$
in $\PP$-probability. 
Hence, by \eqref{E:tau-Kn} there exists $N$ such that for every $n>N$
$$
\mathbb{P}\left( \frac{1}{n} \log m(S(b_-, b_+) \cap \mathcal{G}_{|x_{\tau_n}|})  \leq \epsilon  \textup{ and }|x_{\tau_n}| \geq K_n \right) >  \frac{p}{4},
$$
which implies that $\limsup_n   \frac{1}{n} \log m( S(b_-, b_+) \cap \mathcal{G}_{K_n} ) \leq \epsilon $, hence
$$ \limsup_n   \frac{1}{n} \log^+ m( S(b_-, b_+) \cap \mathcal{G}_{K_n} ) = 0.$$

By Egorov's lemma there exists a set $Z \subseteq B_- \times B_+$ and a sequence $\phi_n$ with $\log^+ \phi_n/n \to 0$ such that
$\lambda_- \otimes \lambda_+(Z) \geq 1 - p/8$ and $m(S(b_-, b_+) \cap \mathcal{G}_{K_n}) \leq \phi_n$ for all $(b_-, b_+) \in Z$.

Then by taking $X_n$ as the intersection of $Z$ and the set from eq. \eqref{E:p5}  we have $\lambda_- \otimes \lambda_+(X_n) \geq p/8$, 
$\widetilde{p}_n^{b^+}[S(b_-, b_+) \cap \mathcal{G}_{K_n}] \geq p/4$ and $m(S(b_-, b_+) \cap \mathcal{G}_{K_n}) \leq \phi_n$
for all $(b_-, b_+) \in X_n$. 
Now, let us define 
$$V := \bigcap_n \bigcup_{m \geq n} X_m$$
and let $W$ be the projection of $V$ to $B_+$. Then for a.e. $b_+ \in W$ there exists $b_-$ such that $(b_-, b_+) \in X_n$ for infinitely many $n$. 
Define $A_{n, \epsilon}^{b^+} := S(b_-, b_+) \cap \mathcal{G}_{K_n}$. Then by construction $V \subseteq Z$, hence 
$m(A_{n, \epsilon}^{b^+}) \leq \phi_n$ for any $n$, so 
$$\limsup_n \frac{1}{n} \log m(A_{n, \epsilon}^{b^+}) \leq 0$$
and $\widetilde{p}_n^{b^+}(A_{n, \epsilon}^{b+}) \geq p/4$ for infinitely many values of $n$, hence 
$$\limsup_n \mathbb{P}^{b^+}(x_{\tau_n} \in A_{n, \epsilon}^{b+}) \geq p/4 > 0.$$
Now, by Theorem \ref{thm : spread-out measure trick}, the induced random walk with measure $\theta$ has bounded density and has finite $H_n$ for any $n$; hence, by Theorem \ref{thm : general An sets} the space $(B_+, \lambda_+)$ is the Poisson boundary of $(G, \theta)$. Finally, by Theorem 
\ref{thm : spread-out measure trick}, $(G, \mu)$ and $(G, \theta)$ have the 
same Poisson boundary, hence $(B_+, \lambda_+)$ is the Poisson boundary of $(G, \mu)$. 
\end{proof}

\bibliographystyle{abbrv}
\bibliography{biblography}

\end{document}